\documentclass[12pt,leqno,twoside]{amsart}
\usepackage{amssymb,amsmath,amsthm,soul,color}
\usepackage{t1enc}
\usepackage[cp1250]{inputenc}
\usepackage{a4,indentfirst,latexsym}
\usepackage{graphics}
\usepackage{mathrsfs}
\usepackage{cite,enumitem,graphicx}
\usepackage[colorlinks=true,urlcolor=blue,
citecolor=red,linkcolor=blue,linktocpage,pdfpagelabels,
bookmarksnumbered,bookmarksopen]{hyperref}
\usepackage[english]{babel}
\usepackage[left=2.61cm,right=2.61cm,top=2.72cm,bottom=2.72cm]{geometry}
\usepackage[metapost]{mfpic}
\usepackage[normalem]{ulem}

\newtheorem{Th}{Theorem}[section]

\newtheorem{Lem}[Th]{Lemma}

\newtheorem{Rem}[Th]{Remark}

\newenvironment{altproof}[1]
{\noindent
{\em Proof of {#1}}.}
{\nopagebreak\mbox{}\hfill $\Box$\par\addvspace{0.5cm}}

\newcommand{\wt}{\widetilde}

   \newcommand{\vp}{\varphi}
   
   \newcommand{\eps}{\varepsilon}

   \def\div{\mathop{\mathrm{div}}}

   \def\supp{\mathrm{supp}}

   \def\d{\diamond}

   \def\N{\mathbb{N}}
   \def\R{\mathbb{R}}

   \def\curl{\mathrm{curl}}

   \def\V{\mathcal{V}}

   \def\W{\mathcal{W}}

\def\rh{\rightharpoonup}

\def\rn{\mathbb{R}^N}

\newcommand{\cC}{{\mathcal C}}
\newcommand{\cD}{{\mathcal D}}

\newcommand{\cM}{{\mathcal M}}
\newcommand{\cN}{{\mathcal N}}
\newcommand{\cO}{{\mathcal O}}

\newcommand{\cS}{{\mathcal S}}

\newcommand{\cV}{{\mathcal V}}
\newcommand{\cW}{{\mathcal W}}

\newcommand{\Om}{\Omega}

\def\r{\mathbb{R}}
\def\r3{\mathbb{R}^3}
\def\rn{\mathbb{R}^N}

\def\eps{\varepsilon}
\def\rh{\rightharpoonup}

\def\io{\int_{\Omega}}

\def\ir3{\int_{\r3}}
\def\vp{\varphi}
\def\wt{\widetilde}

\def\ol{\overline}

\def\lam{\lambda}
\def\lam1{\lambda_1}

\def\d1p{\mathcal{D}^{1,p}}

\def\curlop{\nabla\times}
\newcommand{\weakto}{\rightharpoonup}
\newcommand{\pa}{\partial}

\newcommand{\tu}{\widetilde{u}}
\newcommand{\tv}{\widetilde{v}}
\newcommand{\tcV}{\widetilde{\cV}}

\numberwithin{equation}{section}

\title[Nonlinear curl-curl problem with critical exponent]{
A Sobolev-type inequality for the curl operator and ground states for the curl-curl equation with critical Sobolev exponent}

\author[J. Mederski]{Jaros\l aw Mederski}
\author[A. Szulkin]{Andrzej Szulkin}

\address[J. Mederski]{\newline\indent
	{
		Institute of Mathematics,
		\newline\indent 
		Polish Academy of Sciences,
		\newline\indent 
		ul. \'Sniadeckich 8, 00-656 Warsaw, Poland
		\newline\indent 
		and
		\newline\indent 
		CRC 1173 Wave phenomena: Analysis and Numerics,
		\newline\indent 
		Departement of Mathematics,
		Karlsruhe Institute of Technology (KIT), 
		\newline\indent 
		D-76128 Karlsruhe, Germany}
}
\email{\href{mailto:jmederski@impan.pl}{jmederski@impan.pl}}

\address[A. Szulkin]{\newline\indent 
	Department of Mathematics,
	\newline\indent 
	Stockholm University,
	\newline\indent 
	106 91 Stockholm, Sweden
	}
\email{\href{mailto:andrzejs@math.su.se}{andrzejs@math.su.se}}


\begin{document}

\begin{abstract} 
Let $\Omega\subset \mathbb{R}^3$ be a Lipschitz domain and let $S_\mathrm{curl}(\Omega)$ be the largest constant such that 
$$
\int_{\mathbb{R}^3}|\nabla\times u|^2\, dx\geq S_{\mathrm{curl}}(\Omega) \inf_{\substack{w\in W_0^6(\mathrm{curl};\mathbb{R}^3)\\ \nabla\times w=0}}\Big(\int_{\mathbb{R}^3}|u+w|^6\,dx\Big)^{\frac13}
$$
for any $u$ in $W_0^6(\mathrm{curl};\Omega)\subset W_0^6(\mathrm{curl};\mathbb{R}^3)$ where $W_0^6(\mathrm{curl};\Omega)$ is the closure of $\mathcal{C}_0^{\infty}(\Omega,\mathbb{R}^3)$ in $\{u\in L^6(\Omega,\mathbb{R}^3): \nabla\times u\in L^2(\Omega,\mathbb{R}^3)\}$ with respect to the norm $(|u|_6^2+|\nabla\times u|_2^2)^{1/2}$. We show that  $S_{\mathrm{curl}}(\Omega)$ is strictly larger than the classical Sobolev constant $S$ in $\mathbb{R}^3$. Moreover, $S_{\mathrm{curl}}(\Omega)$ is independent of $\Omega$  and is attained by a ground state solution to the curl-curl problem
$$
\nabla\times (\nabla\times u) = |u|^4u
$$
if  $\Omega=\mathbb{R}^3$. With the aid of these results 
we also investigate ground states of the Brezis-Nirenberg-type problem for the curl-curl operator in a bounded domain $\Omega$
$$\nabla\times (\nabla\times u) +\lambda u = |u|^4u\quad\hbox{in }\Omega$$
with the so-called metallic boundary condition
$\nu\times u=0$ on $\partial\Omega$, where $\nu$ is the exterior normal to $\partial\Omega$. 
\end{abstract}

\maketitle

\section{Introduction} \label{sec:intro}

Sobolev-type inequalities have been widely studied by a large number of authors and the best Sobolev constants play an important role e.g. in the theory of partial differential equations, differential geometry, isoperimetric inequalities as well as in mathematical physics, see e.g. \cite{Lieb,Aubin,Talenti}. In particular, if $\Om$ is a domain in $\r3$, then the best constant $S$ in the Sobolev inequality
\begin{equation} \label{se}
\int_{\Om}|\nabla u|^2\,dx\geq S \Big(\int_{\Om}|u|^6\,dx\Big)^{\frac13}\quad\hbox{ for }u\in\cD^{1,2}(\Om)
\end{equation}
has been computed explicitly by Talenti \cite{Talenti} and as is well-known, it is achieved (i.e., equality holds) if and only if $\Om=\r3$ and $u$ is the Aubin-Talenti instanton $U_{\eps,y}(x) := 3^{1/4}(\eps^2+|x-y|^2)^{-1/2}$, see \cite{Aubin,Talenti}. When $\eps=1$, this is the unique (up to translations in $\r3$) positive solution to the equation $-\Delta u=|u|^4u$ in $\cD^{1,2}(\R^3)$ and a ground state, i.e.  a minimizer for the energy functional among all nontrivial solutions.

 The aim of this work is to perform a similar analysis for the curl operator $\curlop(\cdot)$. This is challenging from the mathematical point of view and important in mathematical physics; such operator appears e.g. in Maxwell equations as well as in Navier-Stokes problems \cite{Dautray,Monk,Girault}. Finding a formulation in the spirit of \eqref{se}, but involving the curl operator, is not straightforward and there are several essential difficulties as we shall see later.                                                                                                                   
 
 For instance, the kernel of $\curlop(\cdot)$ is of infinite dimension since $\curlop(\nabla \vp)=0$ for all $\vp\in\cC^2(\Om)$. Hence the inequality \eqref{se} with $\nabla u$ replaced by $\curlop u$ would hold for all $u\in \cC_0^\infty(\r3,\r3)$ only if $S=0$. This makes it necessary to introduce a Sobolev-like constant in a different way which we now proceed to do. 

 Let $\Om$ be a Lipschitz domain in $\r3$ and for $2\le p\le 6$, let
$$
W^p(\curl;\Om) := \{u\in L^p(\Om,\R^3): \curlop u\in L^2(\Om,\R^3)\}.
$$
This is a Banach space if provided with the norm
$$
\|u\|_{W^p(\curl;\Om)} := \left(|u|^2_p+|\curlop u|^2_2\right)^{1/2}.
$$
Here and in the sequel $|\cdot|_q$ denotes the $L^q$-norm for $q\in[1,\infty]$. 
We also define
\begin{equation} \label{wzerosix}
W^p_0(\curl;\Om) := \text{closure of } \cC^{\infty}_0(\Om,\R^3) \text{ in }W^p(\curl;\Om).
\end{equation}
If $\Om=\R^3$, these two spaces coincide, see Lemma \ref{density}. Although results of this kind are well known, we provide a proof for the reader's convenience. The spaces $W^2(\curl;\Om)$ and $W^2_0(\curl;\Om)$ are studied in detail in \cite{Dautray,KirschHettlich,Monk}.
Extending $u\in W^p_0(\curl;\Om)$ by 0 outside $\Om$ we may assume $W^p_0(\curl;\Om)\subset W^p_0(\curl;\r3)$.
Denote the kernel of $\curlop (\cdot )$ in $W^6_0(\curl;\r3)$ by
$$\cW:=\{w\in W^6_0(\curl;\r3): \curlop w=0\}.$$
Let $S_{\curl}(\Om)$ be the largest possible constant such that the inequality
\begin{equation}\label{eq:neq}
\int_{\r3}|\curlop u|^2\, dx\geq S_{\curl}(\Om)\inf_{w\in \cW}\Big(\int_{\r3}|u+w|^6\,dx\Big)^{\frac13}
\end{equation}
holds for any $u\in W^6_0(\curl;\Om)\setminus \W$. Inequality \eqref{eq:neq} is in fact (trivially) satisfied also for $u\in W^6_0(\curl;\Om)\cap\cW$ because then both sides are zero. Note that here $u$ but not necessarily $w$ is supported in $\Om$.  It is not a priori clear that $S_{\curl}(\Om)$ is positive or that it is independent of $\Om$. That this is the case follows from Theorems \ref{TheoremS_curl} and \ref{Th:main1}(a) below:

\begin{Th}\label{TheoremS_curl}
$S_{\curl}(\Om) = S_\curl$ where $S_\curl := S_\curl(\r3)$.
\end{Th}

In the next result we show that $S_{\curl}$ is attained provided $\Om=\R^3$ and  the optimal function is (up to rescaling) a ground state solution to the curl-curl problem with  critical exponent. Existence of a ground state in this case has been an open question for some time. 
Let
\begin{equation}\label{eq:action}
	J(u):=\frac12\int_{\R^3} |\curlop u|^2\,dx- \frac16\int_{\R^3} |u|^6\, dx
	\end{equation}
and introduce the following constraint:
\begin{equation}\label{def:Neh}
	\cN:=\Big\{u\in W^6_0(\curl;\R^3)\setminus\cW: \int_{\R^3}|\curlop  u|^2=\int_{\R^3}|u|^6\, dx\hbox{ and }\div(|u|^4u)=0\Big\}.
	\end{equation}
As we shall see later, this set is a variant of a generalization of the Nehari manifold \cite{Nehari2} which may be found in \cite{Pankov} for a Schr\"odinger equation.
\begin{Th}\label{Th:main1}
(a) $S_\curl> S$. \\
(b)  $\inf_{\cN}J=\frac13S_{\curl}^{3/2}$ and is attained. Moreover, if $u\in\cN$ and $J(u)=\inf_{\cN} J$, then $u$ is a ground state solution to the equation
\begin{equation}
	\label{eq}
	\curlop (\curlop u)= |u|^4u\quad \hbox{in }\R^3
	\end{equation}
	and
	equality holds in \eqref{eq:neq} for this $u$. If $u$ satisfies equality in \eqref{eq:neq}, then there are unique $t>0$ and $w\in\cW$ such that $t(u+w)\in\cN$ and $J(t(u+w))=\inf_{\cN} J$.\\
\end{Th}

A natural question arises whether ground states must have some symmetry properties. It follows from Theorem 1.1 in \cite{Bartsch} that any $\cO(3)$-equivariant (weak) solution to \eqref{eq} is trivial, hence a ground state cannot be radially symmetric.

The curl-curl problem $\curlop (\curlop u)= f(x,u)$ in a bounded domain or in $\R^3$ has been recently studied e.g. in \cite{Bartsch,BenFor,BartschMederski1,BartschMederski2,Mederski,MSchSz} under different hypotheses on $f$ but always assuming $f$ is subcritical, i.e. $f(x,u)/|u|^5\to 0$ as $|u|\to\infty$. However, the occurence of ground states to \eqref{eq} (i.e., in the critical exponent case) has been an open problem as we have already mentioned.  In view of the existence of Aubin-Talenti instantons, this is a very natural question. While the instantons are given explicitly, we have no such explicit formula for ground states in the curl-curl case. Since the instantons are radially symmetric up to translations, one can find them by ODE methods. In view of the above remark  concerning $\cO(3)$-equivariant solutions, such methods do not seem available for the curl-curl problem and a different approach is needed. Note further that there is no maximum principle for the curl-curl operator and, to our knowledge, no unique continuation principle applicable to our case. An approach different than for \eqref{se} is also required for the proof of $\Om$-independence of $S_{\curl}$, see Section \ref{sec:nonex}. Moreover concentration-compactness analysis for the curl operator is considerably different from that in \cite{ev,Lions85,Willem} -- see our approach in Section \ref{sec:con-com}.

We would like to emphasize an important role of the analysis of nonlinear curl-curl problems from the physical point of view. Solutions $u$ to nonlinear curl-curl equations describe the profiles of time-harmonic solutions $E(x,t)=u(x)\cos(\omega t)$ 
to the time-dependent nonlinear electromagnetic wave equation, which together with material constitutive laws and Maxwell equations, describes the {\em exact} propagation of electromagnetic waves in a nonlinear medium \cite{Agrawal,BartschMederski1,Stuart:1993}. Since finding propagation exactly may be very difficult, there are several simplifications in the literature which rely on  approximations of the nonlinear electromagnetic wave equation. The most prominent one is the scalar or vector nonlinear Schr\"odinger equation. 
For instance, one assumes that the term $\nabla(\div(u))$ in $\curlop(\curlop u)= \nabla(\div(u))-\Delta u$ is negligible and can be dropped, or one uses the so-called {\em slowly varying envelope approximation}. However, such simplifications may produce {\em non-physical} solutions; see \cite{Akhmediev-etal, Ciattoni-etal:2005} and the references therein.

We also point out that the term $|u|^4u$ in \eqref{eq} as well as in \eqref{eq:main} below allows to consider the so-called quintic effect in nonlinear optics modelled by Maxwell equations. See for instance \cite{Agrawal,BartschMederski1,Stuart:1993,Doerfler,MederskiJFA2018,Desyatnikov,Mihalache} and the references therein. We hope that our results will prompt further analytical studies of physical phenomena involving the quintic  nonlinearity, e.g. the well-known cubic-quintic effect in nonlinear optics \cite{Desyatnikov,Mihalache}. 

Using our concentration-compactness result we are also able to treat the Brezis-Nirenberg problem \cite{BrezisNirenberg} for the curl-curl operator
\begin{equation}\label{eq:main}
\curlop(\curlop u) + \lambda u = |u|^{4}u \qquad\textnormal{in } \Om
\end{equation}
 together with the so-called metallic boundary condition
\begin{equation}\label{eq:bc}
\nu\times u = 0\qquad\text{on }\pa\Om.
\end{equation}
Here $\nu:\pa\Om\to\R^3$ is the exterior normal and $\Om\subset\r3$ is a bounded domain. This boundary condition is natural in the theory of Maxwell equations and it holds when $\Om$ is surrounded by a perfect conductor.
If the boundary of $\Om$ is not of class $\cC^{1}$, then we assume \eqref{eq:bc} is satisfied in a generalized sense by which we mean $u$ is in the space $W^6_0(\curl;\Om)$ defined in \eqref{wzerosix}.  Weak solutions to \eqref{eq:main}--\eqref{eq:bc} correspond to critical points 
of the associated energy functional $J_\lambda: W^6_0(\curl;\Om)\to \R$ given by
\begin{equation}\label{eq:defOfJ}
J_\lambda(u):=\frac12\int_\Om |\curlop u|^2\, dx + \frac{\lambda}{2}\int_{\Om} |u|^2\, dx-\frac{1}{6}\int_\Om |u|^6\, dx.
\end{equation} 
Recall  from \cite{MederskiJFA2018,BartschMederski2} that the spectrum of the curl-curl operator in $H_0(\curl;\Om) := W^2_0(\curl;\Om)$
consists of the eigenvalue $\lambda_0=0$ with infinite multiplicity and of a sequence of eigenvalues 
$$0<\lambda_1\le \lambda_2\leq\dots\le\lambda_k\to\infty$$ with corresponding finite multiplicities $m(\lambda_k)\in\N$.
Let $\cN_\lambda$ be the generalized Nehari manifold for $J_\lambda$ (see \eqref{eq:DefOfN} for the definition), and for $\lambda\le 0$ let
\[
c_\lambda := \inf_{\cN_\lambda} J_\lambda.
\]
Denote the Lebesque measure of $\Om$ by $|\Om|$. We introduce the following condition:
\begin{itemize}
\item[($\Om$)] $\Om$ is a bounded domain, either convex or with $\cC^{1,1}$-boundary. 
\end{itemize}
The reason for this assumption will be explained in the next section.

In domains $\Om\ne \r3$ we also introduce another constant, $\ol S_\curl(\Om)$, such that the inequality
\begin{equation}\label{eq:neqOm}
\int_{\Om}|\curlop u|^2\, dx\geq \ol S_\curl(\Om)\inf_{w\in \cW_{\Om}}\Big(\int_{\Om}|u+w|^6\,dx\Big)^{\frac13}
\end{equation}
holds for any $u\in W^6_0(\curl;\Om)\setminus \W_\Om$, where
$\cW_\Om:=\{w\in W^6_0(\curl;\Om): \curlop w=0\}$, and $\ol S_\curl(\Om)$ is largest with this property. As in \eqref{eq:neq}, also here the above inequality trivially holds if $u\in\cW_\Om$. Although $\ol S_\curl(\Om)$ seems to be more natural than $S_\curl(\Om)$, we do not know whether it equals $S_\curl$. We are only able to prove the following result:

\begin{Th}\label{th:Sbar} Let $\Omega$ be a Lipschitz domain in $\r3$, possibly unbounded, $\Omega\ne \r3$. Then $S_\curl\geq  \ol S_\curl(\Om)$. If $\Om$ satisfies $(\Om)$, then $\ol S_\curl(\Om)\ge S$. 
\end{Th}

Finally, the main result concerning the Brezis-Nirenberg problem for the curl-curl operator \eqref{eq:main} reads as follows.

\begin{Th}\label{thm:mainBN} 
	Suppose $\Om$ satisfies $(\Om)$. Let  $\lambda\in(-\lambda_\nu, -\lambda_{\nu-1}]$ for some $\nu\geq 1$.  
	Then $c_\lambda >0$ and
	 the following statements hold:\\
	(a) If $c_\lambda<c_0$, then there is a  ground state solution to \eqref{eq:main}--\eqref{eq:bc}, i.e. $c_\lambda$ is attained by a critical point of $J_\lambda$. A sufficient condition for this inequality to hold is $\lambda\in(-\lambda_\nu,-\lambda_\nu+ \ol S_\curl(\Om)|\Om|^{-2/3})$. \\
	(b) There exists $\eps_\nu\geq \ol S_\curl(\Om)|\Om|^{-2/3}$ such that $c_\lambda$ is not attained for $\lambda\in(-\lambda_{\nu}+\eps_\nu,-\lambda_{\nu-1}]$, and $c_\lambda=c_0$  for $\lambda\in[-\lambda_{\nu}+\eps_\nu,-\lambda_{\nu-1}]$. We do not exclude that $\eps_\nu>\lambda_\nu-\lambda_{\nu-1}$, so these intervals may be empty.\\
	(c) $c_\lambda\to 0$ as $\lambda\to -\lambda_\nu^-$, and the function $$(-\lambda_{\nu},-\lambda_{\nu}+\eps_\nu]\cap (-\lambda_{\nu},-\lambda_{\nu-1}]\ni\lambda\mapsto c_\lambda\in (0,+\infty)$$ is continuous and strictly increasing.
	\\
	(d) There are at least $\#\big\{k: -\lambda_k <\lambda <-\lambda_k +\frac13\ol S_\curl(\Om)|\Om|^{-\frac23}\big\}$ pairs of solutions $\pm u$ to \eqref{eq:main}--\eqref{eq:bc}. 
\end{Th}

Note that if $\lambda$ is as in (a), then the relation $-\lambda_k <\lambda <-\lambda_k +\frac13\ol S_\curl(\Om)|\Om|^{-\frac23}$ holds for $k=\nu,\ldots,\nu+m-1$ where $m$ is the multiplicity of $\lambda_\nu$ but it may also hold for some $k$ with $\lambda_k>\lambda_\nu$.

The above result is known for cylindrically symmetric domains where it is possible to reduce the curl-curl operator to a positive definite one, see \cite{MederskiJFA2018}. However, the solution obtained there is a ground state in a subspace of functions having cylindric symmetry and we do not know whether it is a ground state in the full space.

Let us recall from earlier work that the main difficulties when treating $J$ and $J_\lambda$, also in the subcritical case, are that these functionals are strongly indefinite, i.e., they are unbounded from above and from below, even on subspaces of finite codimension. Moreover, the quadratic part of $J$ has infinite-dimensional kernel and $J', J'_\lambda$ are not (sequentially) weak-to-weak$^*$ continuous, i.e. $u_n\weakto u$ does not imply that $J'_\lambda(u_n)\vp\to J'_\lambda(u)\vp$ for all $\vp\in \cC_0^{\infty}(\Om,\R^3)$. This lack of continuity is caused by the fact that $W^p_0(\curl;\Om)$ is not (locally) compactly embedded in any Lebesgue space and we do not know whether necessarily $u_n\to u$ a.e. in $\Om$. A consequence of this is that for a Palais-Smale sequence $u_n\rh u$ it is not clear whether $u$ is a critical point.  In the subcritical case  one can overcome these difficulties since either a variant of the Palais-Smale condition is satisfied or some compactness  can be recovered on a suitable topological manifold, see e.g. \cite{BartschMederski1,Mederski,MSchSz}. In the critical case however, there are additional difficulties. In Section \ref{sec:con-com} we introduce a general concentration-compactness analysis for this case. We show that the topological manifold 
\begin{equation*}
\Big\{u\in W_0^6(\curl;\R^3): \div(|u|^4u)=0\Big\}
\end{equation*}
is locally compactly embedded in $L^p(\R^3,\R^3)$ for $1\leq p<6$ and that if a sequence $(u_n)$ is contained in this manifold and $u_n\rh u$, then $u_n\to u$ a.e. after passing to a subsequence. This result will play a crucial role  because it implies that if such $(u_n)$ is a Palais-Smale sequence, then $u$ is a solution for our equation. If the condition $\div(|u|^4u)=0$ is violated, the embedding need not be locally compact.

The paper is organized as follows. In Section \ref{sec:setting} we introduce the functional setting and some notation. Section \ref{sec:con-com} concerns the concentration-compactness analysis as we have already mentioned. In Section \ref{sec:Om=R3} we prove Theorem \ref{Th:main1}, and in Section \ref{sec:nonex} we  prove Theorems \ref{TheoremS_curl} and \ref{th:Sbar}. The proof of Theorem \ref{thm:mainBN} is contained in Section \ref{sec:bn} whereas in Section \ref{op} we state some open problems. 

\section{Functional setting and preliminaries}\label{sec:setting}

Throughout the paper we assume that $\Om$ is a Lipschitz domain in $\r3$ and $2\le p\leq 2^*=6$. 
The curl of $u$, $\curlop u$, should be understood in the distributional sense.
We shall look for solutions to \eqref{eq} and \eqref{eq:main}--\eqref{eq:bc} in the space $W_0^6(\curl;\r3)$ and $W^6_0(\curl;\Om)$ respectively. 
We introduce the subspaces
\begin{eqnarray*}
\cV_\Om &:=&\left\{v\in W^6_0(\curl;\Om): \int_\Om\langle v,\varphi\rangle\,dx=0
\text{ for every $\varphi\in \cC^\infty_0(\Om,\R^3)$ with $\curlop\varphi=0$} \right\},\\
\cW_\Om &:=& \left\{w\in W^6_0(\curl;\Om): \int_\Om\langle w,\curlop\varphi\rangle \,dx = 0
\text{ for all }\varphi\in\cC^\infty_0(\Om,\R^3)\right\}\\
&=& \{w\in W^6_0(\curl;\Om): \curlop w=0 \text{ in the sense of distributions}\}.
\end{eqnarray*}
The second one has already been defined in Section \ref{sec:intro}. Here and below $\langle .\,,.\rangle$ denotes the inner product in $\r3$. If $\Om=\r3$, we shall usually write $\cV$ and $\cW$ for $\cV_{\r3}$ and $\cW_{\r3}$.

In the sequel $\Om$ is always a Lipschitz domain and $C$ denotes a generic positive constant which may vary from one equation to another.

In the following subsections we consider two cases.

\subsection{$\Om=\R^3$}

\begin{Lem} \label{density}
	$W^p(\curl;\R^3) = W_0^p(\curl;\R^3)$ for each $p\in[2,6]$.
\end{Lem}

\begin{proof}
	We show $\cC_0^\infty(\R^3,\R^3)$ is dense in $W^p(\curl;\R^3)$. Let $\chi_R\in \cC_0^\infty(\R^3)$ be such that $|\nabla\chi_R|\le 2/R$, $\chi_R=1$ for $|x|\le R$ and $\chi_R= 0$ for $|x|\ge 2R$. Take $u=(u_1,u_2,u_3)\in W^p(\curl;\R^3)$. Then $\chi_Ru\to u$ in $L^p(\R^3,\R^3)$ as $R\to\infty$. We have 
	\begin{equation} \label{chi}
	\partial_i(\chi_Ru_j)-\partial_j(\chi_Ru_i) = (\partial_i\chi_R)u_j - (\partial_j\chi_R)u_i + \chi_R(\partial_iu_j-\partial_ju_i), \quad i\ne j.
	\end{equation}
If $p=2$, it is clear that $(\partial_i\chi_R)u_j\to 0$ in $L^2(\r3)$. If $2<p\le 6$, then 
\[
\ir3(\partial_i\chi_R)^2u_j^2\,dx \le \left(\int_{R\le|x|\le 2R}|\partial_i\chi_R|^q\,dx\right)^{2/q}\left(\int_{R\le|x|\le 2R}|u_j|^p\,dx\right)^{2/p} 
\]
where $q=2p/(p-2)\ge 3$. Since
\[
\int_{R\le|x|\le 2R}|\partial_i\chi_R|^q\,dx \le CR^{3-q}<+\infty,
\]
also here $(\partial_i\chi_R)u_j\to 0$ in $L^2(\r3)$. As $\partial_iu_j-\partial_ju_i\in L^2(\R^3)$, it follows that the left-hand side in \eqref{chi} tends to $\partial_iu_j-\partial_ju_i$ in $L^2(\R^3)$ as $R\to\infty$. 
Hence $\chi_R u \to u$ in $W^p(\curl;\R^3)$ and functions of compact support are dense in $W^p(\curl;\R^3)$.
		
	 Suppose now  $u\in W^p(\curl;\R^3)$ has a compact support. Clearly, $j_\eps*u\to u$ in $L^p(\R^3,\R^3)$ as $\eps\to 0$ where $j_\eps$ is the standard mollifier. Since
	\begin{equation} \label{mollify}
	\partial_i(j_\eps*u_j)-\partial_j(j_\eps*u_i) = j_\eps*(\partial_iu_j-\partial_ju_i)
	\end{equation}
	and $\partial_iu_j-\partial_ju_i\in L^2(\R^3)$, the right-hand side above tends to $\partial_iu_j-\partial_ju_i$ in $L^2(\R^3)$ as $\eps\to 0$. This completes the proof.
\end{proof}

As usual, let $\cD^{1,2}(\R^3,\R^3)$ denote the completion of $\cC^{\infty}_0(\R^3,\R^3)$ with respect to the norm $|\nabla \cdot|_2$.
The following Helmholtz decomposition holds, see \cite{Mederski,MSchSz}:

\begin{Lem}\label{defof} $\V$ and $\cW$ are closed subspaces of $W^6_0(\curl;\R^3)$  and
	\begin{equation}\label{HelmholzDec}
	W^6_0(\curl;\R^3)=\V\oplus \W.
	\end{equation}
	Moreover, $\cV\subset\cD^{1,2}(\r3,\r3)$ and the norms $|\nabla \cdot|_2$ and $\|\cdot\|_{W^6(\curl;\R^3)}$ are equivalent in $\V$.
\end{Lem}

We note that $\cW$ is the closure of $\{\nabla\vp: \vp\in \cC_0^\infty(\r3)\}$. Indeed, if $w\in\cW$, then $\curlop w = 0$, hence we can find $\vp_n$ such that $\nabla\vp_n\to w$ and $\nabla\vp_n\in \cC_0^\infty(\r3,\r3)$ \cite{Mederski,MSchSz}. Since $\nabla\vp_n=0$ outside of some ball, $\vp_n$ is constant there and we may assume this constant is 0. 

\subsection{$\Om$ bounded} \label{sectionOm}

Recall $H_0(\curl;\Om) := W^2_0(\curl;\Om)$ and  note that
\[
\begin{aligned}
\cV_\Om
&\subset\left\{u\in H_0(\curl;\Om): \div(u)\in L^2(\Om,\R^3)\right\}.
\end{aligned}
\]
Here we have used the fact that if $\vp$ in the definition of $\cV_\Om$ is supported in a ball, then $\vp=\nabla\psi$ for some $\psi$ and hence $u\in\cV_\Om$ implies $\div(u)=0$. 
It follows from \cite{Amrouche,Costabel}  that 
$\cV_\Om$
is continuously embedded in $H^{s}(\Om,\R^3)$ for some $s\in [1/2,1]$, hence compactly in $L^2(\Om,\R^3)$. If, in addition $\Om$ satisfies $(\Om)$, then $\cV_\Om$ is continuously embedded in $H^{1}(\Om,\R^3)$, hence compactly in $L^p(\Om,\R^3)$ for $1\leq p<6$ and continuously in $L^6(\Om,\R^3)$.
This implies in particular that 
\begin{equation} \label{vomega}
\cV_\Om=\left\{v\in H_0(\curl;\Om): \int_\Om\langle v,\varphi\rangle\,dx=0
\text{ for every $\varphi\in \cC^\infty_0(\Om,\R^3)$ with $\curlop\varphi=0$} \right\}
\end{equation}
is a Hilbert space with inner product
\[
(v,z) = \int_{\Om}\langle \curlop v,\curlop z\rangle\,dx \equiv \io  \langle\nabla v,\nabla z\rangle\,dx.
\]
 Observe that the right-hand side of \eqref{vomega} is a closed linear subspace of $W^6_0(\curl;\Om)$ as a consequence of $(\Om)$. 
Using this, it follows from the decomposition in \cite[Theorem 4.21(c)]{KirschHettlich} that also here there is a Helmholtz type decomposition
$$W^6_0(\curl;\Om) = \cV_\Om\oplus\cW_\Om$$ and that
\begin{equation*}
\int_{\Om}\langle v,w\rangle \,dx = 0 \quad \text{if } v\in \cV_\Om,\ w\in \cW_\Om  
\end{equation*}
which means that $\cV_\Om$ and $\cW_\Om$ are orthogonal in $L^2(\Om,\R^3)$. In $W_0^6(\curl;\Om)=\V_\Om\oplus\W_\Om$ we can use the norm
$$\|v+w\|:=\big((v,v)+|w|_6^2\big)^{\frac12}, \quad v\in \cV_\Om,\ w\in \cW_\Om$$
which is equivalent to $\|\cdot\|_{W_0^6(\curl;\Om)}$ if ($\Om$) is satisfied. 

 According to \cite[Theorem IX.2]{Dautray} or \cite[Theorem~3.33]{Monk}, there is a continuous tangential trace operator $\gamma_t:H(\curl;\Om) := W^2(\curl;\Om) \to H^{-1/2}(\pa\Om)$ such that
$$
\gamma_t(u)=\nu\times u|_{\pa\Om}\qquad\text{for any $u\in \cC^{\infty}(\overline\Om,\R^3)$}
$$
and
$$
H_0(\curl;\Om)=\{u\in H(\curl;\Om): \gamma_t(u)=0\}.
$$
Hence any vector field $u\in  W^6_0(\curl;\Om)=\V_\Om\oplus\W_\Om\subset H_0(\curl;\Om)$ satisfies the metallic boundary condition \eqref{eq:bc}.

 Denote the subspace of all gradient vector fields in $W_0^{1,6}(\Om)$ by $\nabla W_0^{1,6}(\Om)$. Clearly, 
$\nabla W_0^{1,6}(\Om)\subset\cW_\Om$. 
However, for general domains the subspace $\{w\in \W_\Om: \div(w)=0\}$ may be nontrivial and hence  $\nabla W^{1,6}_0(\Om)\subsetneq\cW_\Om$, see \cite[pp. 4314 and 4315]{BartschMederski2} and \cite[Theorem 3.42]{Monk}.

\begin{Lem}\label{lem:W_Om}
There holds $\cW_\Om=W^6_0(\curl;\Om)\cap \cW= W^6_0(\curl;\Om)\cap\nabla W^{1,6}(\Om)$.
If $\partial\Om$ is connected, then $\cW_\Om=\nabla W^{1,6}_0(\Om)$. If $\Om$ is unbounded, $\cW_\Om=W^6_0(\curl;\Om)\cap \cW$ still holds.
\end{Lem}

\begin{proof}
	Let $w\in\cW_{\Om}$ and take a sequence $(\vp_n)\subset\cC_0^{\infty}(\Om,\R^3)$ such that $\vp_n\to w$ in $W_0^6(\curl;\Om)$. Extend $\vp_n$ by $0$ in $\R^3\setminus\Om$ and note that $(\vp_n)$ is a Cauchy sequence, so $\vp_n\to \wt w$ in $W^6_0(\R^3,\R^3)$ where $\wt w|_\Om=w$ and $\wt w=0$ in $\r3\setminus\Om$.
As 
	$$\int_{\R^3}\langle \wt w,\curlop \psi\rangle\,dx=\lim_{n\to\infty}
	\int_{\R^3}\langle \vp_n,\curlop \psi\rangle\,dx=
	\lim_{n\to\infty}
	\int_{\R^3}\langle \curlop\vp_n,\psi\rangle\,dx\leq 
	\lim_{n\to\infty} |\curlop \vp_n|_2|\psi|_2=0$$
	for any $\psi\in \cC_0^{\infty}(\R^3,\R^3)$, it follows that $ \wt w\in\cW$. Moreover, since $\wt w\in L^6(\R^3,\R^3)$ and $\curlop \wt w =0$, in view of \cite[Lemma 1.1]{Le} we obtain $\wt w=\nabla \psi$ for some $\psi \in W^{1,6}_{loc}(\R^3)$. Therefore $w=\nabla \psi|_{\Om}\in\nabla W^{1,6}(\Om)$. Clearly, $W^6_0(\curl;\Om)\cap \cW$ and  $W^6_0(\curl;\Om)\cap\nabla W^{1,6}(\Om)$ are contained in $\cW_\Om$.
	
	Suppose that $\partial\Om$ is connected. Similarly as above, we obtain $w=\nabla \psi$ for some $\psi \in W^{1,6}(\Om)$ and the surface gradient
	$$\nabla_S \psi=(\nu \times \nabla \psi)\times \nu=0.$$
	Therefore we may assume that $\psi\in W^{1,6}_0(\Om)$, cf. \cite[Theorem 4.3 and Remark 4.4]{Monk}.
\end{proof}

\section{General concentration-compactness analysis in $\rn$}
\label{sec:con-com}

In this, self-contained, section we have $N\ge 3$ and we work in subspaces of $L^{2^*}(\rn,\rn)$ where $2^* := 2N/(N-2)$.

Let $\Om$ be a domain in $\rn$, $\cV$ a closed subspace of $\cD^{1,2}(\R^N,\R^N)$ and 
\begin{equation}\label{def:W}
\W:=\big\{ w=(w_1,...,w_N)\in L^{2^*}(\Om,\R^N): \curlop w = 0\big\}
\end{equation}
where $\curlop w$ denotes the skew-symmetric, matrix-valued distribution having $\partial_k w_l-\partial_l w_k\in \cD'(\Om)$ as matrix elements. So for $N=3$, $\cW$ corresponds to $\cW_\Om$ in Section \ref{sec:setting} but $\cV$ may be a more general subspace. Note that $\curlop $ is the usual curl operator if $N=3$. Let $Z$ be a finite-dimensional subspace of $L^{2^*}(\Om,\R^N)$ such that $Z\cap \cW = \{0\}$ and put 
\[
\wt{\cW} := \cW\oplus Z.
\]
Assume
\begin{itemize}
	\item[(F1)] $F:\Om\times\R^N\to\R$ is differentiable with respect to the second variable $u\in\R^N$ for a.e. $x\in\Om$, $F(x,0)=0$ and $f=\pa_uF:\Om\times\R^N\to\R^N$ is a Carath\'eodory function (i.e., $f$ is measurable in $x\in\Om$ for all $u\in\rn$ and continuous in $u\in\R^N$ for a.e.\ $x\in\Om$);
	\item[(F2)] $F$ is uniformly strictly convex with respect to $u\in\R^N$, i.e.\ for any compact set $A\subset(\R^N\times\R^N)\setminus\{(u,u):\;u\in\R^N\}$
	$$
	\inf_{\genfrac{}{}{0pt}{}{x\in\Om}{(u_1,u_2)\in A}}
	\left(\frac12\big(F(x,u_1)+F(x,u_2)\big)-F\left(x,\frac{u_1+u_2}{2}\right)\right) > 0;
	$$
	\item[(F3)] There are $c_1,c_2>0$ and $a\in L^{N/2}(\Om)$, $a\ge 0$, such that
	$$c_1|u|^{2^*}\leq F(x,u)\quad\hbox{and }|f(x,u)|\le a(x)|u|+c_2|u|^{2^*-1}$$
	for every $u\in\R^N$ and a.e. $x\in\Om$.
\end{itemize}

In view of (F2) and (F3), for any $v\in\cV$ we find a unique $\wt w_\Om(v)\in \wt{\cW}$ such that
\begin{equation}\label{eq:ineqF}
\int_{\Om}F(x,v+\wt w_\Om(v))\,dx \le \int_{\Om}F(x,v+\wt w)\,dx \quad \text{for all } \wt w\in \wt\cW.
\end{equation}
This implies that
\begin{equation} \label{eq:eqf}
\int_{\Om}\langle f(x,v+\wt w),\zeta\rangle\,dx  =0 \quad \text{for all } \zeta\in\wt\cW \text{ if and only if } \wt w=\wt w_\Om(v).
\end{equation}
Denote the space of finite measures in $\rn$ by $\cM(\rn)$.

\begin{Th}\label{Th:Concentration}
	Assume that (F1)--(F3) are satisfied. 
	Suppose $(v_n)\subset\cV$,  $v_n\weakto v_0$ in $\cV$, $v_n\to v_0$ a.e. in $\rn$, $|\nabla v_n|^2\weakto \mu$ and $|v_n|^{2^*}\rh \rho$
	 in $\cM(\R^N)$. Then
	there exists an at most countable set $I\subset\R^N$ and nonnegative weights $\{\mu_x\}_{x\in I}$, $\{\rho_x\}_{x\in I}$ such that
	$$
	\mu\geq |\nabla v_0|^2+\sum_{x\in I}\mu_x\delta_x, \quad \rho = |v_0|^{2^*} + \sum_{x\in I}\rho_x\delta_x,
	$$
	 and passing to a subsequence, $\wt w_\Om(v_n)\rh \wt w_\Om(v_0)$ in $\wt\cW$, $\wt w_\Om(v_n)\to \wt w_\Om(v_0)$ a.e. in $\Om$ and in $L^p_{loc}(\Om)$ for any $1\leq  p<2^*$.
\end{Th}

\begin{Rem}  \emph{
We shall use this theorem in Sections \ref{sec:Om=R3} and \ref{sec:bn}. In Section \ref{sec:Om=R3} we have $\Om=\r3$ and $Z=\{0\}$, so $\wt w = w$ and we will write $w(v)$ for $w_{\r3}(v)$.  In Section \ref{sec:bn}, where we treat a Brezis-Nirenberg problem, $\Om$ will be bounded and $Z$ the subspace of $\cV_{\Om}$ on which the quadratic part of $J_\lambda$ (see \eqref{eq:defOfJ}) is negative semidefinite.
}
\end{Rem} 

\medskip

\begin{altproof}{Theorem \ref{Th:Concentration}}
{\em Step 1.} Let $\vp\in\cC_0^{\infty}(\R^N)$. By the Sobolev inequality,  
	\begin{eqnarray}
	\Big(\int_{\R^N}|\vp|^{2^*}|v_n-v_0|^{2^*}\,dx\Big)^{1/2^*}\label{eq:standardLions}
	&\leq&
	S^{-1/2}\Big(\int_{\R^N}|\nabla [\vp(v_n-v_0)]|^{2}\,dx\Big)^{1/2} \\ \nonumber
	&=&
	S^{-1/2}\Big(\int_{\R^N}|\vp|^2|\nabla (v_n-v_0)|^{2}\,dx\Big)^{1/2}+o(1).
	\end{eqnarray}
	Passing to the limit and using the Brezis-Lieb lemma \cite{BrezisLieb,Willem} on the left-hand side above we obtain
	\begin{equation} \label{eq:standardLions2}
	\Big(\int_{\R^N}|\vp|^{2^*}d\bar\rho\Big)^{1/2^*} \le S^{-1/2}\Big(\int_{\R^N}|\vp|^2\,d\bar{\mu}\Big)^{1/2}
	\end{equation}
	where 
	$\bar{\mu}:=\mu-|\nabla v_0|^2$ and $\bar{\rho}:=\rho-|v_0|^{2^*}$. Set $I=\{x\in\R^N: \mu_x:=\mu(\{x\})>0\}$. Since $\mu$ is finite and $\mu, \bar\mu$ have the same singular set,  $I$ is at most countable and $\mu\geq |\nabla v_0|^2+\sum_{x\in I}\mu_x\delta_x$. As in the proof of Theorem 1.9 in \cite{ev} it follows from \eqref{eq:standardLions2} that $\bar\rho = \sum_{x\in I}\rho_x\delta_x$, see also Proposition 4.2 in \cite{wa}. So $\mu$ and $\rho$ are as claimed.
	 
	 {\em Step 2.} Using (F3) and \eqref{eq:ineqF} we infer that
	\begin{eqnarray*}
	c_1|v_n+\wt w_\Om(v_n)|_{2^*}^{2^*}&\leq& \int_{\Om} F(x,v_n+\wt w_\Om(v_n))\leq \int_{\Om} F(x,v_n)\,dx \\
	&\leq& c_2 |v_n|_{2^*}^{2^*} + |a|_{N/2}|v_n|_{2^*}^2,
	\end{eqnarray*}
	and since the right-hand side above is bounded, so is $(|\wt w_\Om(v_n)|_{2^*})$. Hence, up to a subsequence, $\wt w_\Om(v_n)\rh \wt w_0$ for some $\wt w_0$. Write $\wt w_\Om(v_n) = w_n+z_n$, $\wt w_0=w_0+z_0$ where $w_n,w_0\in\cW$ and $z_n,z_0\in Z$.  We shall show that $\wt w_\Om(v_n)\to \wt w_0$ a.e. in $\Om$ after taking subsequences. Obviously, we may assume  $z_n\to z_0$ in $Z$ and a.e. in $\Om$.
	
	  We can find a sequence of open balls $(B_{l})_{l=1}^\infty$
	such that $\Om=\bigcup_{l=1}^\infty B_l$. Fix $l\geq 1$.
	In view of \cite[Lemma 1.1]{Le} there exists $\xi_n\in W^{1,2^*}(B_l)$ such that $w_n=\nabla\xi_n$ and we may assume without loss of generality that $\int_{B_l}\xi_n\,dx = 0$. Then by the Poincar\'e inequality,
	\[
	\|\xi_n\|_{W^{1,2^*}(B_l)} \le C|w_n|_{L^{2^*}(B_l,\R^N)} \le C|w_n|_{2^*}
	\]
	and passing to a subsequence, $\xi_n\rh \xi$ for some $\xi\in W^{1,2^*}(B_l)$. So $\xi_n\to \xi$ in $L^{2^*}(B_l)$. 
	Now take any $\vp\in\cC_0^{\infty}(B_l)$.
	Since $\nabla (|\vp|^{2^*}(\xi_n-\xi))\in\cW$, in view of \eqref{eq:eqf} we get
	$$\int_{\Om}\langle f(x,v_n+\wt w_\Om(v_n)),\nabla (|\vp|^{2^*} (\xi_n-\xi))\rangle\,dx = 0,$$
	that is, 
	\[
	\int_{\Om}|\vp|^{2^*}\langle f(x,v_n+\wt w_\Om(v_n)),w_n-\nabla\xi\rangle \,dx 
	=\int_{\Om}\langle f(x,v_n+\wt w_\Om(v_n)),\nabla (|\vp|^{2^*})(\xi-\xi_n)\rangle\,dx,
	\]
	where the right-hand side tends to $0$ as $n\to\infty$. Since $w_n\rh \nabla\xi$ in $L^{2^*}(B_l)$,  
	$$
	\int_{\Om}|\vp|^{2^*}\langle f(x,v_0+\nabla\xi+z_0),\, w_n-\nabla\xi\rangle \,dx =o(1),
	$$
	hence, recalling that $\wt w_\Om(v_n)=w_n+z_n$ and $z_n\to z_0$, we obtain
	\begin{equation} \label{eq:est1}
	\int_{\Om}|\vp|^{2^*}\langle f(x,v_n+\wt w_\Om(v_n))-f(x,v_0+\nabla\xi+z_0),\, \wt w_\Om(v_n)-\nabla\xi-z_0\rangle \,dx =o(1).
	\end{equation}
	The convexity of $F$ in $u$ implies that 
	$$F\Big(x,\frac{u_1+u_2}{2}\Big)\geq F(x,u_1)+ \Big\langle f(x,u_1),\,\frac{u_2-u_1}{2}\Big\rangle$$
	and
	$$F\Big(x,\frac{u_1+u_2}{2}\Big)\geq F(x,u_2)+\Big\langle f(x,u_2),\, \frac{u_1-u_2}{2}\Big\rangle.$$
	Adding these inequalities and using (F2), we obtain for any $k\geq 1$ and $|u_1-u_2|\ge \frac1k$, $|u_1|, |u_2| \le k$ that
	\begin{equation} \label{mk}
	m_{k} \le \frac12(F(x,u_1)+F(x,u_2)) - F\Big(x,\frac{u_1+u_2}2\Big)\le \frac14\langle f(x,u_1)-f(x,u_2),\, u_1-u_2\rangle
	\end{equation}
	where
	\begin{equation} \label{mrR}
	m_{k}:=\inf_{\substack{x\in\Om,u_1,u_2\in\R^N\\ \frac1k\leq|u_1-u_2|,\\|u_1|,|u_2|\leq k} }\;
	\frac{1}{2}(F(x,u_1)+F(x,u_2))-F\Big(x,\frac{u_1+u_2}{2}\Big)>0.
	\end{equation} 
	Let 
	$$\Omega_{n,k}:=\Big\{x\in \Om: |v_n+\wt w_\Om(v_n)-v_0-\nabla\xi-z_0|\geq  \frac1k \hbox{ and }|v_n+\wt w_\Om(v_n)|,|v_0+\nabla\xi+z_0|\leq k\Big\}.$$
	Taking into account \eqref{eq:est1} and using (F3), \eqref{mk} and H\"older's inequality, we get
	\begin{eqnarray*}
		&&4m_k\int_{\Omega_{n,k}}|\vp|^{2^*}\,dx \\ 
		&& \quad \le \int_{\Om}|\vp|^{2^*}\langle f(x,v_n+\wt w_\Om(v_n))-f(x,v_0+\nabla\xi+z_0), \, v_n+\wt w_\Om(v_n)-v_0-\nabla\xi-z_0\rangle \,dx\\
		&& \quad \le \int_{\Om}|\vp|^{2^*}\langle f(x,v_n+\wt w_\Om(v_n))-f(x,v_0+\nabla\xi+z_0), \, v_n-v_0\rangle \,dx +o(1)\\
		&& \quad \le C \Big(\int_{\Om}|\vp|^{2^*}|v_n-v_0|^{2^*}\,dx\Big)^{1/2^*} +o(1) = 
		C \Big(\int_{\Om}|\vp|^{2^*}\,d\bar{\rho}\Big)^{1/2^*}+o(1),
	\end{eqnarray*}
where $k$ is fixed. Here we have used the fact that $\io a(x)|v_n-v_0|^2\,dx \to 0$ if $v_n\rh v_0$ in $L^{2^*}(\Om,\rn)$.  
	Since $\vp\in\cC_0^{\infty}(B_l)$ is arbitrary, 
	\begin{equation}\label{eq:Borel}
	4m_k|\Omega_{n,k}\cap E|\leq \big(\bar{\rho}(E)\big)^{1/2^*}+o(1)
	\end{equation}
	for any Borel set $E\subset B_l$. We find an open set $E_k\supset I$ such that $|E_k|<1/2^{k+1}$. Then, taking $E=B_l\setminus E_k$ in \eqref{eq:Borel}, we have
	$4m_k|\Omega_{n,k}\cap (B_l\setminus E_k)|=o(1)$ as $n\to\infty$ because $\supp(\bar\rho)\subset I$; hence 
	we can find a sufficiently large $n_k$ such that $|\Omega_{n_k,k}\cap B_l|<1/2^k$ and we obtain
	$$\Big|\bigcap_{j=1}^{\infty}\bigcup_{k=j}^{\infty}\Omega_{n_k,k}\cap B_l\Big| \le \lim_{j\to\infty}\sum_{k=j}^{\infty}|\Omega_{n_k,k}\cap B_l|\leq \lim_{j\to\infty}\frac{1}{2^{j-1}}=0.$$
	If $x\notin \bigcap_{j=1}^{\infty}\bigcup_{k=j}^{\infty}\Omega_{n_k,k}$ and $x\in B_l$, then
	\begin{eqnarray*}
	&&|v_{n_k}(x)+\wt w_\Om(v_{n_k})(x)-v_0(x)-\nabla\xi(x)-z_0(x)|< \frac1k,\hbox{ or }|v_{n_k}(x)+\wt w_\Om(v_{n_k})(x)|>k,\\
	&&\hbox{ or }|v_0(x)+\nabla\xi(x)+z_0(x)|> k
	\end{eqnarray*}
for all sufficiently large $k$. Since $v_{n_k}+\wt w_\Om(v_{n_k})$ is bounded in $L^{2^*}(\Om,\R^N)$, the second and the third inequality above cannot hold on a set of positive measure for all large $k$. 
We infer that $v_{n_k}+\wt w_\Om(v_{n_k})\to v_0+\nabla\xi+z_0$, hence $\wt w_\Om(v_{n_k})\to \nabla\xi+z_0$  a.e. in $B_l$. Since $\wt w_\Om(v_n)\weakto \wt w_0$, $\wt w_0=\nabla \xi+z_0$ a.e. in $B_l$. Now employing the diagonal procedure, we find a subsequence of $\wt w_\Om(v_n)$ which converges to $\wt w_0$ a.e. in $\Om=\bigcup_{l=1}^{\infty}B_l$. 

Let $p\in [1,2^*)$. For $\Omega'\subset \Omega$ such that $|\Omega'|<+\infty$ we have
	\begin{equation*}
	\int_{\Om'} |v_n-v_0+\wt w_\Om(v_n)-\wt w_0|^p\,dx
	\leq |\Omega'|^{1-\frac{p}{2^*}}
	\Big(\int_{\Om}|v_n-v_0+\wt w_\Om(v_n)-\wt w_0|^{2^*}\,dx\Big)^{\frac{p}{2^*}},
	\end{equation*}
	hence by the Vitali convergence theorem, $v_n-v_0+\wt w_\Om(v_n)-\wt w_0\to 0$ in $L^p_{loc}(\Omega)$ after passing to a subsequence.
	
	{\em Step 3.} We show that $\wt w_\Om(v_0)=\wt w_0$.
	Take any $\wt w\in\wt{\cW}$ and observe that by  the Vitali convergence theorem, 
	$$0 = \int_{\Om}\langle f(x,v_n+\wt w_\Om(v_n)), \wt w\rangle\,dx\to \int_{\Om}\langle f(x,v_0+\wt w_0), \wt w\rangle\,dx$$
up to a subsequence. Now \eqref{eq:eqf} implies that $\wt w_0=\wt w_\Om(v_0)$ which completes the proof.
\end{altproof}

\section{Problem in $\Om=\R^3$ and proof of Theorem \ref{Th:main1}} \label{sec:Om=R3}

Let $S$ be the  best Sobolev constant  for the embedding of $\cD^{1,2}(\R^3)$ into $L^6(\R^3)$, see \eqref{se}. It is clear that a minimizer $w(u)$ in \eqref{eq:ineqF} exists uniquely for any $u\in W_0^6(\curl;\Om)$, not only for $u\in \cV$. Here we have $F(x,u)=\frac16|u|^6$ and $Z=\{0\}$. So by Lemma \ref{defof}, $u+w(u)=v+w(v)\in \cV\oplus\cW$ for some $v\in\cV$ and therefore
\begin{equation} \label{eqeq}
\inf_{w\in\cW} \ir3 |u+w|^6\,dx = \ir3|u+w(u)|^6\,dx = \ir3|v+w(v)|^6\,dx.
\end{equation}
Since $\div(v)=0$,
\begin{equation} \label{eq:scurl}
S_\curl = \inf_{\substack{u\in W_0^6(\curl;\r3) \\ \curlop u\neq 0}} \frac{|\nabla\times u|_2^2}{|u+w(u)|_6^2} = \inf_{v\in \cV\setminus\{0\}} \frac{|\nabla v|_2^2}{|v+w(v)|_6^2}.
\end{equation}

\begin{Lem}\label{lemmaScurlS}
$S_{\curl}\ge S$.
\end{Lem}

\begin{proof}
Given $\eps>0$, by \eqref{eq:scurl} we can find $v\ne 0$ such that
\begin{equation} \label{eq:5}
\int_{\R^3}|\nabla v|^2\, dx\le (S_{\curl}+\eps)\Big( \int_{\R^3}|v+w(v)|^6\,dx\Big)^{\frac13}.
\end{equation}
Let $v=(v_1,v_2,v_3)$. By the H\"older inequality,
\begin{equation} \label{eq:3}
\ir3 v_1^2v_2^2v_3^2\,dx \le \Big(\ir3v_1^6\,dx \ir3v_2^6\,dx \ir3v_3^6\,dx\Big)^{\frac13}
\end{equation}
and
\begin{equation} \label{eq:4}
\ir3 v_i^4v_j^2\,dx \le \Big(\ir3 v_i^6\Big)^{\frac23}\Big(\ir3 v_j^6\,dx\Big)^{\frac13}, \quad i\ne j.
\end{equation}
Using this and the Sobolev inequality gives
\begin{equation}\label{eq:1}
\int_{\R^3}|\nabla v|^2\,dx\geq S\sum_{i=1}^3\Big(\int_{\R^3}|v_i|^6\,dx\Big)^{1/3}\geq S\Big(\int_{\R^3}|v|^6\,dx\Big)^{1/3},
\end{equation}
and since $w(v)$ is a minimizer, we obtain using \eqref{eq:5} and \eqref{eq:1} 
\begin{eqnarray} \label{eq:2}
\int_{\R^3}|\nabla v|^2\, dx & \le & (S_{\curl}+\eps) \Big(\int_{\R^3}|v+w(v)|^6\,dx\Big)^{\frac13} \le (S_{\curl}+\eps) \Big(\int_{\R^3}|v|^6\,dx\Big)^{\frac13} \\
& \le & (S_\curl+\eps)/S\int_{\R^3}|\nabla v|^2\, dx. \nonumber
\end{eqnarray}
Hence $S_{\curl}+\eps\geq S$ for all $\eps>0$ and the conclusion follows.
\end{proof}

Next we look for ground states for the  curl-curl problem \eqref{eq}, i.e. nontrivial solutions
with  least possible associated  energy  $J$ given by \eqref{eq:action}. Throughout the rest of the paper we shall make repeated use of the following fact:

\begin{Lem} \label{stretching}
Let $\lambda>0$. Then $w(\lambda u) = \lambda w(u)$. Similarly, if $\Om$ is a proper subset of $\r3$, then $w_\Om(\lambda u) = \lambda w_\Om(u)$.
\end{Lem}

\begin{proof}
We prove this for $w_\Om$. Using the minimizing property of $w_\Om(u)$ we obtain
\begin{eqnarray*}
\lambda^6\io|u+w_\Om(u)|^6\,dx & = & \io|\lambda u+\lambda w_\Om(u)|^6\,dx \ge \io|\lambda u+w_\Om(\lambda u)|^6\,dx \\
& = & \lambda^6\io |u+w_\Om(\lambda u)/\lambda|^6\,dx \ge \lambda^6\io|u+w_\Om(u)|^6\,dx.
\end{eqnarray*}
Since the minimizer is unique, $w_\Om(u) = w_\Om(\lambda u)/\lambda$ as claimed.
\end{proof}

\begin{Lem} \label{equivnehari}
Let $\cN$ be the set defined in \eqref{def:Neh}. Then 
\begin{equation} \label{descrneh}
\cN =\{u\in W^6_0(\curl;\R^3)\setminus\cW: J'(u)u=0 \text{ and } J'(u)|_\cW=0\}. 
\end{equation}
\end{Lem}

\begin{proof}
The first condition in \eqref{def:Neh} is equivalent to $J'(u)u=0$. The second condition is satisfied because $\div(|u|^4u)=0$ if and only if $\ir3\langle |u|^4u,\nabla\vp\rangle\,dx = 0$ for all $\vp\in \cC_0^\infty(\r3)$ and each element of $\cW$ can be approximated by such $\vp$, see the comment preceding Subsection \ref{sectionOm}.  
\end{proof}

By Lemma \ref{defof}, $W_0^6(\curl;\R^3) = \cV\oplus\cW$.
It follows from \eqref{eq:ineqF} and \eqref{eq:eqf} that if $v\in\cV$, then $J'(v+w(v))|_\cW=0$, and as 
\begin{equation} \label{eq:jtu}
J(t(v+w(v))) = \frac{t^2}2\ir3|\nabla v|^2\,dx - \frac{t^6}6\ir3|v+w(v)|^6\,dx,
\end{equation}
there is a unique $t(v)>0$ such that
\begin{equation} \label{eq:m}
m(v):=t(v)(v+w(v))\in\cN\quad\hbox{for }v\in \cV\setminus\{0\}.
\end{equation}
We note that
\begin{equation} \label{maximizing}
J(m(v)) \ge J(t(v+w)) \quad \text{for all } t>0 \text{ and }w\in\cW.
\end{equation}
Since $J(m(v)) \ge J(v)$ and there exist  $a,r>0$ such that $J(v)\ge a$ if $\|v\|=r$, $\cN$ is bounded away from $\cW$ and hence closed.

\begin{Lem}\label{lem:m_continuous}
The mapping $m: \cV\setminus\{0\}\to\cN$ given by \eqref{eq:m} is continuous. 
\end{Lem} 

\begin{proof}
Let $v_n\to v_0\ne 0$ in $\cV$. Since 
\begin{equation} \label{eq:ab}
\ir3|v_n+w(v_n)|^6\,dx \le \ir3|v_n|^6\,dx, 
\end{equation}
it follows that $(w(v_n))$ is bounded and it is then clear from \eqref{eq:jtu} that so is $(t(v_n))$. Hence we may assume $t(v_n)\to t_0$ and $w(v_n)\weakto w_0$ in $L^6(\R^3,\R^3)$.
By the  weak sequential lower semicontinuity of  the second integral in \eqref{eq:jtu} and by \eqref{maximizing}, 
$$
J(t_0(v_0+w_0))\geq \limsup_{n\to\infty} J(t(v_n)(v_n+w(v_n)))\geq \limsup_{n\to\infty} J(t_0(v_n+w_0))
=J(t_0(v_0+w_0)).
$$
So $w(v_n)\to w_0$ and since $\cN$ is closed, $t_0(v_0+w_0) = t(v_0)(v_0+w(v_0)) = m(v_0)$.
\end{proof}

Now it is easily seen  that $m|_\cS:\cS:=\{v\in\cV: \|v\|=1\}\to\cN$ is a homeomorphism with the inverse $u=v+w(v)\mapsto v/\|v\|$.
 Note that $\cN$ is an infinite-dimensional topological manifold of infinite codimension.
Although  $J$ is of class $\cC^2$, we do not know whether $\cN$ is of class $\cC^1$. However, repeating the argument  in \cite[Proposition 4.4(b)]{Mederski} or \cite[Proposition 2.9]{SzulkinWeth} we see that $J\circ m|_{\cS}:\cS\to\R$ is of class $\cC^1$  and is bounded from below by the constant $a>0$ introduced above. By the Ekeland variational principle \cite[Theorem 8.5]{Willem}, there is a Palais-Smale sequence $(v_n)\subset \cS$ such that 
\begin{equation} \label{gea}
(J\circ m)(v_n)\to \inf_{\cS} J\circ m=\inf_{\cN}J\geq a>0.
\end{equation}
It follows from \cite[Proposition 4.4(b)]{Mederski} again or from \cite[Corollary 2.10]{SzulkinWeth}
that $(m(v_n))$ is a Palais-Smale sequence for $J$ on $\cN$, so in particular, $J'(m(v_n))\to 0$ as $n\to\infty$. See also an abstract critical point theory on the generalized Nehari manifold in \cite[Section 4]{BartschMederski1} and  in\cite[Section 4]{BartschMederski2}.

For $s>0$, $y\in\R^3$ and $u:\R^3\to\R^3$ we denote $T_{s,y}(u):= s^{1/2}u(s\cdot +y))$. 
The following lemma is a special case of \cite[Theorem 1]{Solimini}, see also \cite[Lemma 5.3]{Tintarev}.

\begin{Lem}\label{lem:Solimini}
Suppose that $(v_n)\subset \cD^{1,2}(\r3,\r3)$ is bounded. Then $v_n\to 0$ in $L^6(\R^3,\R^3)$ if and only if $T_{s_n,y_n}(v_n)\weakto 0$ in $\cD^{1,2}(\r3,\r3)$ for all $(s_n)\subset\R^+$ and $(y_n)\subset \R^3$.
\end{Lem}

Observe that the above lemma in \cite{Solimini} is expressed in terms of the space $H^{1,2}$. However, in the notation of \cite{Solimini}, this is the same space as our  $\cD^{1,2}$.

\begin{Lem} \label{isom}
$T_{s,y}$ is an isometric isomorphism of $W_0^6(\curl;\r3)$ which leaves the functional $J$ and the subspaces $\cV, \cW$ invariant. In particular, $w(T_{s,y}u)=T_{s,y}w(u)$.
\end{Lem}

The proof is by an explicit (and simple) computation.

\begin{Lem} \label{equiv}
Suppose $u+w(u)\in \cN$. Then 
\[
\frac{|\nabla\times u|_2^2}{|u+w(u)|_6^2} = A \quad \text{if and only if} \quad J(u+w(u)) = \frac13 A^{3/2}.
\]
In particular, $\inf_\cN J = \frac13S_\curl^{3/2}$.
\end{Lem}

\begin{proof}
Since $u+w(u)\in \cN$, $J'(u)u=0$, i.e. $|\nabla\times u|_2^2 = |u+w(u)|_6^6$. Hence 
\[
\frac{|\nabla\times u|_2^2}{|u+w(u)|_6^2} =  |u+w(u)|_6^4 \quad \text{and} \quad J(u+w(u)) = \frac13|u+w(u)|_6^6.
\] 
\end{proof}

\medskip 

\begin{altproof}{Theorem \ref{Th:main1}} We prove part (b) first. 
Take a minimizing sequence $(u_n) = (m(v_n))\subset \cN$  constructed above and write $u_n=t(v_n)(v_n+w(v_n)) = v_n'+w(v_n')\in\cV\oplus\cW$. As
\begin{equation} \label{ac}
J(u_n) = J(u_n) -\frac16J'(u_n)u_n = \frac13|\curlop u_n|^2_2 = \frac13|\nabla v_n'|^2_2
\end{equation}
and $|\nabla\cdot|_2$ is an equivalent norm in $\cV$, $(v_n')$ is bounded. We also have
\begin{equation} \label{ca}
J(u_n) = J(u_n) -\frac12J'(u_n)u_n  = \frac13|u_n|^6_6.
\end{equation} 
Since $J(u_n)$ is bounded away from 0, $|u_n|_6\not\to 0$ and hence by \eqref{eq:ab}, $|v_n'|_6\not\to 0$. 
Therefore, passing to a subsequence and using Lemma \ref{lem:Solimini}, $\tv_n:=T_{s_n,y_n}(v_n')\weakto v_0$ for some $v_0\neq 0$, $(s_n)\subset\R^+$ and $(y_n)\subset \R^3$. Taking subsequences again we also have that $\tv_n\to v_0$ a.e. in $\R^3$ and in view of Theorem \ref{Th:Concentration}, $w(\tv_n)\weakto w(v_0)$ and $w(\tv_n)\to w(v_0)$ a.e. in $\R^3$. We set $u:=v_0+w(v_0)$ and by Lemma \ref{isom} we may assume without loss of generality that $s_n=1$ and $y_n=0$. So if $z\in W_0^6(\curl;\r3)$, then using weak and a.e. convergence,
\[
J'(u_n)z = \ir3\langle\curlop u_n, \curlop z\rangle\,dx -\ir3\langle|u_n|^4u_n,z\rangle\,dx \to J'(u)z.
\]
Here we have used that $|u_n|^4u_n\rh \zeta$ in $L^{6/5}(\r3,\r3)$ for some $\zeta$ but since $|u_n|^4u_n\to |u|^4u$ a.e., $\zeta = |u|^4u$. So $u$ is a solution to \eqref{eq}. To show it is a ground state, we note that using Fatou's lemma,
\begin{eqnarray*}
\inf_\cN J & = & J(u_n)+o(1) = J(u_n)-\frac12J'(u_n)u_n+o(1) = \frac13|u_n|^6_6 + o(1) \\
& \ge & \frac13|u|^6_6 + o(1) = J(u)-\frac12J'(u)u+o(1) = J(u)+o(1).
\end{eqnarray*}
Hence $J(u)\le \inf_\cN J$ and as a solution, $u\in\cN$. It follows using Lemma \ref{equiv} that $J(u) = \inf_\cN J = \frac13S_\curl^{3/2}$. 

If $u$ satisfies equality in \eqref{eq:neq}, then $t(u)(u+w(u))\in\cN$ and is a minimizer for $J|_\cN$. But then the corresponding point $v$ in $\cS$ is a minimizer for $J\circ m|_\cS$, see \eqref{gea}. So $v$ is a critical point of $J\circ m|_\cS$ and $m(v)=u$ is a critical point of $J$. This completes the proof of (b).

\medskip

(a) By Lemma \ref{lemmaScurlS}, $S_\curl \ge S$ and by part (b), there exists $u=v+w(v)$ for which $S_\curl$ is attained. Suppose $S_\curl=S$. Then all inequalities become equalities in \eqref{eq:2} with $\eps=0$, and therefore also in \eqref{eq:1}.  But then $\ir3|\nabla v_i|^2\,dx = S|v_i|_6^2$ for $i=1,2,3$ and hence all $v_i$ are instantons, up to multiplicative constants. Since $v\ne 0$ and $\div(v)=0$, this is impossible. It follows that $S_\curl>S$. 
\end{altproof}

\section{Proof of Theorems \ref{TheoremS_curl} and \ref{th:Sbar}} \label{sec:nonex}

Let $\Omega$ be a Lipschitz domain in $\r3$. Recall from Section \ref{sec:setting} that we have the Helmholtz decompositions  
\begin{equation} \label{split}
W_0^6(\curl;\r3) = \cV\oplus \cW \quad  \text{and}\quad W_0^6(\curl;\Om) = \cV_\Om\oplus \cW_\Om  
\end{equation}
where the second one holds if condition $(\Om)$ in the introduction is satisfied.
For $u\in W_0^6(\curl;\Om)$, denote the minimizer of 
\[
\io|u+w|^6\,dx, \quad w\in\cW_\Om
\]
by $w_\Om(u)$ (cf. \eqref{eqeq}) and, according to our notational convention, write $w(u)$ for $w_{\r3}(u)$. Recall from \eqref{eq:neq} the definition of $S_\curl(\Om)$:
\[
\ir3|\nabla\times u|^2\,dx \ge S_\curl(\Om)\inf_{w\in\cW}\left(\ir3|u+w|^6\,dx\right)^{1/3}
\]
where $u\in W_0^6(\curl;\Om)\setminus\cW$ and $S_\curl(\Om)$ is the largest constant with this property.
By \eqref{split} we have $u=v+w\in \cV\oplus\cW$. We emphasize that although $u=0$ in $\r3\setminus\ol\Om$, $v$ and $w$ need not be 0 there.
 Note that $S_\curl({\Om)}$ can be characterized as
\begin{equation} \label{sobolev}
S_\curl(\Om) = \inf_{\substack{u\in W_0^6(\curl;\Om)\\ \curlop u\ne 0}} \sup_{w\in \cW} \frac{|\nabla\times u|_2^2}{|u+w|_6^2} =  \inf_{\substack{u\in W_0^6(\curl;\Om)\\ \curlop u\ne 0}} \frac{|\nabla\times u|_2^2}{|u+w(u)|_6^2} 
\end{equation}
(cf. \eqref{eq:scurl}). In domains $\Om\ne \r3$ there is also another constant, $\ol S_\curl(\Om)$, introduced in \eqref{eq:neqOm}. Similarly as in \eqref{sobolev}, it can be characterized as
\begin{equation} \label{sobolev2}
\ol S_\curl(\Om) = \inf_{\substack{u\in W_0^6(\curl;\Om)\\ \curlop u\ne 0}} \sup_{w\in \cW_\Om} \frac{|\nabla\times u|_2^2}{|u+w|_6^2} =  \inf_{\substack{u\in W_0^6(\curl;\Om)\\ \curlop u\ne 0}} \frac{|\nabla\times u|_2^2}{|u+w_\Om(u)|_6^2}. 
\end{equation}
As we have noticed in the introduction, although this constant seems more natural, we do not know whether it equals $S_\curl$.

\begin{Lem} \label{Continuity}
The mapping $u\mapsto w_\Om(u): L^6(\Om,\r3)\to L^6(\Om,\r3)$ is continuous (\,$\Om=\r3$ is admitted).
\end{Lem}

\begin{proof}
Let $u_n\to u_0$. Since $(w_\Om(u_n))$ is bounded, $w_\Om(u_n)\rh w_0$ after passing to a subsequence. By the maximality and uniqueness of $w_\Om(\cdot)$,
\begin{eqnarray*}
\io|u_0+w_\Om(u_0)|^6\,dx & \le & \io|u_0+w_0|^6\,dx \le\liminf_{n\to\infty}\io|u_n+w_\Om(u_n)|^6\,dx \\
&\le & \liminf_{n\to\infty} \io|u_n+w_\Om(u_0)|^6\,dx = \io|u_0+w_\Om(u_0)|^6\,dx.
\end{eqnarray*}
Hence all inequalities above must be equalities and it follows that $w_0=w_\Om(u_0)$ and $w_\Om(u_n)\to w_\Om(u_0)$.
\end{proof} 

We shall need the following inequality:
\begin{Lem}\label{lemNehariineq}
	If $u\in W^6_0(\curl;\Om)\setminus\{0\}$, $w\in \W_\Om$ and $t\geq 0$, then
	\begin{equation}\label{bbb1}
	J(u)\geq J(tu+w) -J'(u)\left[\frac{t^2-1}{2}u+tw\right].
	\end{equation}
	Moreover,  strict inequality holds unless $t=1$ and $w=0$. (\,$\Om=\r3$ admitted.)
\end{Lem} 

\begin{proof}
	The proof follows a similar argument as in \cite[Proposition 4.1]{Mederski} and \cite[Lemma 4.1]{MederskiJFA2018}. We include it for the reader's convenience. We show that 
	\begin{equation}\label{eq:B3check}
	J(u)-J(tu+w) +J'(u)\left[\frac{t^2-1}{2}u+tw\right]\\
	= \int_{\R^3}\vp(t,x)\,dx\geq 0,
	\end{equation}
	where
	\begin{equation*}
	\vp(t,x):= -\Big\langle |u|^{4}u,\frac{t^2-1}{2}u+tw\Big\rangle - \frac16|u|^6 + \frac16|tu+w|^6.
	\end{equation*}
	An explicit computation using $\curlop w = 0$ shows that both sides of \eqref{eq:B3check} are equal. Clearly, $\vp(t,x)\ge 0$ if $u(x)=0$. 
	So let $u(x)\neq 0$. It is easy to check that $\vp(0,x) > 0$ and $\vp(t,x)\to\infty$ as $t\to\infty$. Note that if $\partial_t\vp (t_0,x)=0$ for some $t_0>0$, then either  $\langle u,\,t_0u+w\rangle =0$ or $|u|=|t_0u+w|$.  In the first case, substituting $-\langle u,w\rangle = t_0|u|^2$, we obtain $\vp(t_0,x) = \big(\frac{t_0^2}2+\frac13\big)|u|^6+\frac16|t_0u+w|^6>0$. In the second case we have, using $-t_0\langle u, w\rangle = \frac{t_0^2-1}2|u|^2+\frac12|w|^2$, that $\vp(t_0,x) = \frac12|u|^4|w|^2\ge 0$. Hence
	 $\vp(t,x)\geq 0$ for all $t\geq0$ and the inequality is strict if  $w\neq 0$. 
	If $w=0$, then $\vp(t,x)=\big(\frac{t^6}{6}-\frac{t^2}{2}+\frac13\big)|u|^6>0$ provided  $t\ne 1$.
\end{proof}

Similarly as in \eqref{descrneh} we introduce the  set
\begin{equation}\label{def:NehOm}
\cN_\Om:=\Big\{u\in W^6_0(\curl;\Om)\setminus\cW_\Om: J'(u)u=0\hbox{ and }J'(u)|_{\cW_\Om}=0\Big\}.
\end{equation}

\medskip

\begin{altproof}{Theorems  \ref{TheoremS_curl} and \ref{th:Sbar}}
Since $tu+w(tu) = t(u+w(u))$ according to Lemma \ref{stretching}, we may assume without loss of generality that $u+w(u)\in \cN$ in \eqref{sobolev} and similarly, $u+w_\Om(u)\in \cN_\Om$ in \eqref{sobolev2}.  According to Lemma \ref{equiv},
$$\inf_{\cN} J|_{W_0^6(\curl;\Om)} = \frac13S_\curl(\Om)^{\frac32},\quad \inf_{\cN_\Om} J = \frac13\ol S_\curl(\Om)^{\frac32},\quad\inf_\cN J = \frac13S_\curl^{\frac32}.$$
In view of Lemma \ref{lem:W_Om}, $\W_\Om\subset \cW$, hence we easily infer from  \eqref{sobolev}, \eqref{sobolev2} that $S_\curl(\Om)\geq \ol S_\curl(\Om)$. As $W_0^6(\curl;\Om)\subset W_0^6(\curl;\r3)$, it follows that
$S_\curl \le S_\curl(\Om)$.

Next we show that $S_\curl(\Om)\leq S_\curl$.
Let $u_0$ be a minimizer for $J$ on $\cN$ provided by Theorem \ref{Th:main1}(b) and find a sequence $(u_n)\subset \cC_0^\infty(\r3,\r3)$ such that $u_n\to u_0$. We can decompose $u_n$ as $u_n=v_n+w_n$, $v_n\in\cV$, $w_n\in\cW$.
Since $u_0=v_0+w(v_0)$ (recall $u_0\in\cN$), $u_n=v_n+w_n\to u_0=v_0+w(v_0)$ and therefore $v_n\to v_0$, $w_n\to w(v_0)$. So $v_0\ne 0$ and $v_n$ are bounded away from $0$ in $L^6(\R^3,\R^3)$.
Assume without loss of generality that $0\in\Om$. There exist $\lambda_n$ such that $\wt u_n$ given by $\wt u_n(x):=\lambda_n^{1/2}u_n(\lambda_nx)$ are supported in $\Om$. Set $\wt w_n := w(\wt u_n)\in \cW$ and choose $t_n$ so that $t_n(\wt u_n+\wt w_n)\in\cN$. Then
\begin{equation}\label{tn}
t_n^2 = \frac{|\nabla\times \tu_n|_2}{|\tu_n+\wt w_n|_6^3}.
\end{equation}
According to Lemma \ref{isom}, $\|\wt u_n\| = \|u_n\|$ and $|\tu_n+\wt w_n|_6 = |u_n+w(u_n)|_6 = |v_n+w(v_n)|_6$. As $(u_n)$ is bounded, so is $(\wt u_n)$ and as $|v_n+w(v_n)|_6\to |v_0+w(v_0)|_6$, $|\wt u_n+\wt w_n|_6$ is bounded away from 0. So $(t_n)$ is bounded. Moreover, $|\wt w_n|_6 = |w(u_n)|_6$ and therefore $(\wt w_n)$ is bounded.
Since $J(\wt u_n)=J(u_n)\to \frac13S_\curl^{3/2}$ and $\|J'(\wt u_n)\| = \|J'(u_n)\|\to 0$, it follows from Lemma \ref{lemNehariineq} that
\begin{eqnarray*}
	\frac13S_\curl^{3/2} = \lim_{n\to\infty} J(\wt u_n) & \ge & \lim_{n\to\infty} \left(J(t_n(\wt u_n+\wt w_n)) -J'(\wt u_n)\left[\frac{t_n^2-1}{2}\wt u_n+t_n^2\wt w_n\right]\right) \\
	& = & \lim_{n\to\infty} J(t_n(\wt u_n+\wt w_n)) \ge \frac13 S_\curl(\Om)^{3/2}.
\end{eqnarray*}
The last inequality follows from Lemma \ref{equiv} and the fact that $\wt u_n$ are as in \eqref{sobolev}, i.e. $\wt u_n\in W_0^6(\curl;\Om)$.

It remains to show that $\ol S_{\curl}(\Om)\geq S$ if $(\Om)$ is satisfied. But this follows by repeating the argument of Lemma \ref{lemmaScurlS} with obvious changes: $S_\curl$ should be replaced by $\ol S_{\curl}(\Om)$, $w(v)$ by $w_\Om(v)$ and the domain of integration should be $\Om$.
\end{altproof}

\begin{Rem}
\emph{
Let $\Om\ne \r3$ and suppose  $S_\curl(\Om)$ is attained by some $u$. Extend $u$ by 0 outside $\Om$. As $S_\curl(\Om)=S_\curl$,  $u$ also solves \eqref{eq} in $\r3$, possibly after replacing $u$ with $\alpha u$ for an appropriate $\alpha>0$. In particular, if $S_\curl(\Om)$ were attained in a bounded $\Om$, this would imply the existence of ground states in $\r3$ which have compact support. To our knowledge, there is no unique continuation principle which could rule out this possibility.
}
\end{Rem} 

In view of this remark we expect that similarly as is the case for the Sobolev constant, $S_\curl$ is attained if and only if $\Om=\r3$. We leave this problem as a conjecture.

\section{The Brezis-Nirenberg type problem and proof of Theorem \ref{thm:mainBN}} \label{sec:bn}

Let $\lambda\leq 0$. In this section $\Om\subset\r3$ is a fixed bounded domain satisfying $(\Om)$ but $\lambda$ will be varying. Therefore we drop the subscript $\Om$ from notation and replace it by $\lambda$ ($J_\lambda$, $\cN_\lambda$ etc.). We also write $\cV,\cW$ for $\cV_\Om, \cW_\Om$. 

Recall from the introduction and Subsection \ref{sectionOm} that the spectrum of the curl-curl operator in $H_0(\curl;\Om)$ consists of the eigenvalue $\lambda_0=0$ whose eigenspace is $\cW$ and of a sequence of eigenvalues 
$$0<\lambda_1 \le \lambda_2\leq\dots\le\lambda_k\to\infty$$ with finite multiplicities $m(\lambda_k)\in\N$. The eigenfunctions corresponding to different eigenvalues are $L^2$-orthogonal and those corresponding to $\lambda_k>0$ are in $\cV$.

For $\lambda\leq 0$ we find two closed and orthogonal subspaces $\cV^+$ and $\tcV$ of $\cV$ such that the quadratic form $Q:\cV\to\R$
given by
\begin{equation*}\label{eq:DefQ}
Q(v):=\int_\Om (|\curlop v|^2+\lambda |v|^2)\, dx \equiv \int_\Om (|\nabla v|^2+\lambda |v|^2)\, dx 
\end{equation*}
is positive definite on $\cV^+$ and negative semidefinite on $\tcV$ where $\mathrm{dim}\,\tcV<\infty$. Writing $u=v+w=v^++\wt v+w\in \cV^+\oplus\wt\cV\oplus \cW$, we have
\[
Q(v)=Q(v^+)+Q(\wt v)
\]
and our functional $J_\lambda$ (see \eqref{eq:defOfJ}) can be expressed as
\begin{equation*} \label{Jlambda}
J_\lambda(u) = \frac12Q(v^+) + \frac12 Q(\wt v) + \frac{\lambda}2\io|w|^2\,dx - \frac16|u|^6\,dx.
\end{equation*}
 We shall use Theorem \ref{Th:Concentration} with 
\[
F(x,u)=\frac16|u|^6 -\frac{\lambda}2|u|^2.
\] 
Here $\wt\cW := \wt\cV\oplus \cW$ (so $Z=\wt\cV$ in the notation of Section \ref{sec:con-com}) and $\wt w = \wt v+w$. $\cV$, and hence $\cV^+$, may be considered, after a proper extension, as closed subspaces of $\cD^{1,2}(\R^3,\R^3)$. Indeed, let $U$ be a bounded domain in $\r3$, $U\supset\ol\Om$. Since $\cV\subset H^1(\Om,\r3)$,  each $v\in\cV$ may be extended to $v'\in H^1_0(U,\r3)$ such that $v'|_\Om=v$. This extension is bounded as a mapping from $\cV$ to $H^1_0(U,\r3)$. Since
	$$\cV':=\big\{v'\in H^1_0(U,\R^3): v'|_{\Om}\in \cV\big\}$$
	is a closed subspace of $H^1_0(U,\R^3)$, and hence of $\cD^{1,2}(\R^3,\R^3)$, we can apply Theorem \ref{Th:Concentration} with $F$ as above and $\cV^+$ replacing $\cV$. 
The generalized Nehari manifold is now given by
 \begin{equation}\label{eq:DefOfN}
 \cN_\lambda:=\{u\in W^6_0(\curl;\Om)\setminus(\tcV\oplus \W):\; J_\lambda'(u)|_{\R u\oplus\tcV\oplus \W}=0\}.
 \end{equation}
 As in Section \ref{sec:Om=R3}, also here it is not clear whether $\cN_\lambda$ is of class $\cC^1$. Setting $m_\lambda(v^+):=v^++\wt w(v^+)$ where $v^+\in \cV^+$ and $\wt w(v^+) \equiv \wt w_\Om(v^+)$ is the minimizer as in \eqref{eq:ineqF}, we have 
 \[
m_\lambda(v^+) := t(v^+)(v^++\wt w(v^+)) \in \cN_\lambda, \quad v^+\in \cV^+\setminus\{0\}
 \]
 (cf. \eqref{eq:m}) and $J_\lambda\circ m_\lambda$ is of class $\cC^1$ on $\cS^+$. Moreover, $m_\lambda|_{\cS^+}$ is a homeomorphism between $\cS^+$ and $\cN_\lambda$. As in \eqref{gea}, we may find a Palais-Smale sequence $(v_n^+)\subset\cS^+$ such that  
 \begin{equation} \label{psseq}
 (J_\lambda\circ m_\lambda)(v_n^+)\to \inf_{\cS^+} J_\lambda\circ m_\lambda = c_\lambda \hbox{ and }J_\lambda'(m_\lambda(v_n^+))\to 0
 \end{equation}
where
 $$c_\lambda:=\inf_{\cN_\lambda}J_\lambda.$$
 Note that 
 $$c_0 = \frac13 \ol S_\curl(\Om)^{3/2} \ge \frac13 S^{3/2}.$$ 
 
\begin{Lem} \label{beta}
	Let $\lambda\in(-\lambda_\nu,-\lambda_{\nu-1}]$ for some $\nu\ge 1$. There holds
	\[
	c_\lambda \le \frac13(\lambda+\lambda_\nu)^{3/2}|\Om| \quad \text{and} \quad c_\lambda<c_0 \text{ if } \lambda<-\lambda_\nu+ \ol S_\curl(\Om)|\Om|^{-2/3}.
	\]
\end{Lem}

\begin{proof}
	The first inequality has been established in \cite[Lemma 4.7]{MederskiJFA2018}. However, for the reader's convenience we include the argument. Let $e_\nu$ be an eigenvector corresponding to $\lambda_\nu$. Then $e_\nu\in\cV^+$. Choose $t>0$, $\wt v\in\wt\cV$ and $w\in\cW$ so that $u= v+w=te_\nu+\wt v+w\in\cN_\lambda$. Since $\lambda_k\le\lambda_\nu$ for $k<\nu$, 
	\begin{eqnarray*}
		c_\lambda & \le & J_\lambda(u) = \frac12\io|\curlop v|^2\,dx +\frac{\lambda}2\io|u|^2\,dx -\frac16\io|u|^6\,dx \\
		& \le & \frac{\lambda_\nu}2\io|v|^2\,dx + \frac{\lambda}2\io|u|^2\,dx -\frac16\io|u|^6\,dx \le \frac{\lambda+\lambda_\nu}2\io|u|^2\,dx -\frac16\io|u|^6\,dx \\
		& \le & \frac{\lambda+\lambda_\nu}2\,|\Om|^{2/3}\left(\io|u|^6\,dx\right)^{1/3} - \frac16\io|u|^6\,dx \le \frac13(\lambda+\lambda_\nu)^{3/2}|\Om|.
	\end{eqnarray*}
	In the last step we have used the elementary inequality $\frac A2t^2-\frac16t^6 \le \frac13A^{3/2}$ ($A>0$).
	
	 Since $c_0 = \frac13\ol S_\curl(\Om)^{3/2}$, the second inequality follows immediately.
\end{proof}

 If $c_\lambda<c_0$, then in view of \cite[Theorem 2.2 (a)]{MederskiJFA2018}  there is a Palais-Smale sequence $(u_n)\subset\cN_\lambda$ such that $J_\lambda(u_n)\to c_\lambda>0$ and $u_n\weakto u_0\neq 0$ in $W^6_0(\curl;\Om)$. It has been unclear so far whether $u_0$ is a critical point of $J_\lambda$. Now we shall show using the concentration-compactness analysis from Section \ref{sec:con-com}  that $u_0$ is not only a solution but even  a ground state for \eqref{eq:main}. The following lemma plays a crucial role.

\begin{Lem}\label{lem:ConvL2}
	If $(u_n)\subset \cN_\lambda$ is bounded, then, passing to a subsequence, $u_n\to u_0$ in $L^2(\Om,\R^3)$ for some $u_0$.
\end{Lem} 

\begin{proof}
Let $u_n = m_\lambda(v_n^+) =  v_n^++\wt w(v_n^+)$. Since $\cV^+$ and $\wt\cW$ are complementary subspaces, $(v_n^+)$ is bounded in $\cV^+$. So passing to a subsequence, $v_n^+\rh v_0^+$ in $\cV^+$, and $v_n^+\to v_0^+$ in $L^2(\Om,\r3)$ and a.e. in $\Om$. Hence by Theorem \ref{Th:Concentration},  $\wt w(v_n^+)\to \wt w(v_0^+)$ in $L^2(\Om,\R^3)$, and therefore also $u_n\to u_0$ there.
\end{proof}

\begin{Lem}{(cf. \cite[Lemma 4.6]{MederskiJFA2018})}  \label{coercive}
$J_\lambda$ is coercive on $\cN_\lambda$.
\end{Lem}

\begin{proof}
Let $(u_n)$ be a sequence in $\cN_\lambda$ such that $J_\lambda(u_n)\le d$. Then
\[
d\ge J_\lambda(u_n) = J_\lambda(u_n) - \frac12 J_\lambda'(u_n)u_n = \frac13\io|u_n|^6\,dx, 
\]
hence $(u_n)$ is bounded in $L^6(\Om,\r3)$, and therefore also in $L^2(\Om,\r3)$. It follows that  
\[
d\ge J_\lambda(u_n) = \frac12Q(v_n^+) + \frac12Q(\wt v_n) + \frac{\lambda}2\io|w_n|^2\,dx - \frac16\io|u_n|^6\,dx 
\]
where the last three terms are bounded (recall $\text{dim}\,\wt\cV<\infty$). Hence also $(v_n^+)$ is bounded. 
\end{proof}

 Let 
\[
N(u) := |u|^4u.
\]
 It is clear that $N: L^6(\Om,\r3) \to L^{6/5}(\Om,\r3)$. We shall need the following version of the Brezis-Lieb lemma:

\begin{Lem} \label{bl}
Suppose $(u_n)$ is bounded in $L^6(\Om,\r3)$ and $u_n\to u$ a.e. in $\Om$. Then 
\[
N(u_n)-N(u_n-u)\to N(u) \quad \text{in } L^{6/5}(\Om,\r3) \text{ as } n\to \infty.
\]
\end{Lem}

\begin{proof}
Since $N(u_n)-N(u_n-u)\to N(u)$ a.e. in $\Om$ and $N(u_n)-N(u_n-u)$ is bounded in $L^{6/5}(\Om,\r3)$,  $N(u_n)-N(u_n-u)\rh N(u)$. We claim that $|N(u_n)-N(u_n-u)|_{6/5}\to |N(u)|_{6/5}$. Using Vitali's convergence theorem we obtain
\begin{gather*}
\io \left| |u_n|^4u_n-|u_n-u|^4(u_n-u)\right|^{6/5}\,dx = \io\int_0^1\frac d{dt}\left| |u_n+(t-1)u|^4(u_n+(t-1)u\right|^{6/5}\,dtdx \\
= \io\int_0^1\frac d{dt}|u_n+(t-1)u|^6\,dtdx = 6\int_0^1\io\langle |u_n+(t-1)u|^4(u_n+(t-1)u),\, u\rangle \,dxdt \\
\to 6\int_0^1\io t^5|u|^6\,dxdt = \io|u|^6\,dx.
\end{gather*}
Hence $N(u_n)-N(u_n-u)$ converges strongly to $N(u)$.
\end{proof}

\begin{Lem}\label{LemPScond}
	Let $\beta<c_0$. Then
	$J_\lambda$ satisfies the $(PS)_\beta$-condition in $\cN_\lambda$, i.e. if $(u_n)\subset \cN_\lambda$, $J_\lambda(u_n)\to\beta$ and $J'_\lambda(u_n)\to 0$ as $n\to\infty$, then $u_n\to u_0\ne 0$ in $W_0^6(\curl;\Om)$ along a subsequence. In particular, $u_0$ is a nontrivial solution for \eqref{eq:main}--\eqref{eq:bc}.
\end{Lem}

\begin{proof}
Let $(u_n)$ be a $(PS)_\beta$-sequence such that $(u_n)\subset \cN_\lambda$. According to Lemma \ref{coercive}, $(u_n)$ is bounded and we may assume $u_n\weakto u_0$ in $W_0^6(\curl;\Om)$. By Lemma \ref{lem:ConvL2}, $u_n\to u_0$ in $L^{2}(\Om,\R^3)$ and hence also a.e. in $\Om$ after passing to a subsequence if necessary. As in the proof of Theorem \ref{Th:main1} in Section \ref{sec:Om=R3} we see  that $J'_\lambda(u_0)=0$, i.e. $u_0$ is a solution for \eqref{eq:main}--\eqref{eq:bc}. According to the Brezis-Lieb lemma \cite{BrezisLieb},
	$$
	\lim_{n\to\infty}\Big(\int_{\Om} |u_n|^{6} \,dx -
	\int_{\Om} |u_n-u_0|^{6} \,dx\Big)
	=\int_{\Om} |u_0|^{6} \,dx,$$
	hence
	\begin{equation}\label{eq:PS1}
	\lim_{n\to\infty}\big(J_\lambda(u_n)-J_\lambda(u_n-u_0)\big)=J_\lambda(u_0)\geq 0,
	\end{equation}
	and by Lemma \ref{bl}, 
	\begin{equation}\label{eq:PS22}
	\lim_{n\to\infty}\big(J_{\lambda}'(u_n)-J_{\lambda}'(u_n-u_0)\big)=J_{\lambda}'(u_0) = 0.
	\end{equation}
	Since $J_\lambda'(u_n)\to 0$ and $u_n\to u_0$ in $L^2(\Om,\R^3)$, 
	\begin{equation}\label{eq:PS2}
	\lim_{n\to\infty} J_{0}'(u_n-u_0) = 0.
	\end{equation}
	Suppose  $\liminf_{n\to\infty}\|u_n-u_0\|>0$. Since $\lim_{n\to\infty}J_{0}'(u_n-u_0)(u_n-u_0)=0$, we infer that 
	$$\liminf_{n\to\infty}|\curlop(u_n-u_0)|_2>0.$$
	Let $u_n-u_0=v_n+\wt w_n\in \cV\oplus\cW$ according to the Helmholtz decomposition in $W_0^6(\curl;\Om)$. If $v_n\to 0$ in $L^6(\Om,\R^3)$, then by \eqref{eq:PS2} we have
	$J_{0}'(u_n-u_0)v_n\to 0$, thus
	\[
|\curlop (u_n-u_0)|_2^2 = |\curlop v_n|_2^2=J_{0}'(u_n-u_0)v_n+\int_{\Om}\langle|u_n-u_0|^4(u_n-u_0),v_n\rangle\,dx \to 0
	\]
as $n\to\infty$ which is a contradiction. Therefore $|v_n|_6$ is bounded away from $0$. Put $w_n:= w(u_n-u_0)\in\cW$. Then $(w_n)$ is bounded and since $u_n-u_0+w_n=v_n+w(v_n)\in\cV\oplus\cW$, $|u_n-u_0+w_n|_6$ is bounded away from 0. Choose $t_n$ so that $t_n(u_n-u_0+w_n)\in \cN_0$   ($\cN_0\equiv\cN_\Om$ in the notation of Section \ref{sec:nonex}). As in \eqref{tn} we have
\[
t_n^2 = \frac{|\nabla\times (u_n-u_0)|_2}{|u_n-u_0+w_n|_6^3},
\]
so $(t_n)$ is bounded.
	Using Lemma \ref{lemNehariineq}, as in the proof of Theorems \ref{TheoremS_curl} and \ref{th:Sbar} we get 
	\begin{equation*}\label{eq:PS3}
	J_{0}(u_n-u_0)\geq J_{0}(t_n(u_n-u_0+w_n))-J_{0}'(u_n-u_0)\Big[\frac{t_n^2-1}{2}(u_n-u_0)+t_n^2w_n\Big],
	\end{equation*}
	so by \eqref{eq:PS2} and since $u_n\to u_0$ in $L^2(\Om,\r3)$,   
		$$\beta =  \lim_{n\to\infty}J_\lambda(u_n-u_0)=
	\lim_{n\to\infty}J_{0}(u_n-u_0)\geq \lim_{n\to\infty}J_{0}(t_n(u_n-u_0+w_n))\geq c_0,$$
	 a contradiction. Therefore, passing to a subsequence, $u_n\to u_0$. Since $u_0\in\cN_\lambda$, $u_0\ne 0$. 
\end{proof}

\begin{altproof}{Theorem \ref{thm:mainBN}}
	(a) It follows from \eqref{psseq} and Lemma \ref{LemPScond} that if $c_\lambda<c_0$, then $c_\lambda$ is attained and hence there exists a ground state solution. By Lemma \ref{beta}, this inequality is satisfied whenever $\lambda\le \lambda_{\nu-1}$ and $\lambda\in (-\lambda_\nu,-\lambda_\nu+ \ol S_\curl(\Om)|\Om|^{-2/3})$.
	
 In view of \cite[Theorem 2.2(b)]{MederskiJFA2018}, the function $(-\lambda_{\nu},-\lambda_{\nu-1}]\ni\lambda\mapsto c_\lambda\in (0,+\infty)$ is non-decreasing, continuous and $c_\lambda\to 0$ as $\lambda\to -\lambda_\nu^- $, and if $c_{\mu_1}=c_{\mu_2}$  for some $-\lambda_{\nu}<\mu_1<\mu_2\leq-\lambda_{\nu-1}$, then 
$c_\lambda$ is not attained for $\lambda\in (\mu_1,\mu_2]$. Hence (b) and (c) follow.

 (d) Since $J_\lambda$ is even and, by Lemma \ref{LemPScond}, satisfies the Palais-Smale condition in $\cN_\lambda$ at any level below $c_0$, then, in view of \cite[Theorem 3.2(c)]{MederskiJFA2018}, $J_\lambda$ has at least $m(\cN_\lambda,c_0)$ pairs of critical points $\pm u$ such that $u\neq 0$ and $c_\lambda\leq J_\lambda(u)<c_0$ where
\begin{equation}\label{eq:DefOfm(N,beta)}
m(\cN_\lambda,c_0):=\sup\{\gamma(J_\lambda^{-1}((0,\beta])\cap \cN_\lambda):  \beta<c_0\}
\end{equation}
and $\gamma$ is the Krasnoselskii genus \cite{Struwe}. This is a consequence of the standard fact that if 
\[
\beta_k := \inf\{\beta\in\R: \gamma(J_\lambda^{-1}((0,\beta])\cap \cN_\lambda) \ge k\},
\]
then there are at least as many pairs of critical points as the number of $k$ for which $(PS)_{\beta_k}$ holds, see e.g. \cite{Struwe}.

 In order to complete the proof we show that 
$$m(\cN_\lambda,c_0)\geq \wt{M}(\lambda) := \#\Big\{k: -\lambda_k <\lambda <-\lambda_k +\frac13\ol S_\curl(\Om)|\Om|^{-\frac23}\Big\}.$$
Let
$$A(\lambda):=\big\{k\geq 1: -\lambda_k < \lambda <-\lambda_k +\frac13\ol S_\curl(\Om)|\Om|^{-\frac23}\hbox{ and }\lambda_k>\lambda_{k-1}\}$$
and observe that 
$$\wt{M}(\lambda)=\sum_{k\in A(\lambda)}m(\lambda_k),$$
where $m(\lambda_k)$ stands for the multiplicity of $\lambda_k$.
For $k\in A(\lambda)$, let $\V(\lambda_k)$ denote the eigenspace corresponding to $\lambda_k$. Then $\textrm{dim}\,\V(\lambda_k)=m(\lambda_k)$.
Let 
$S(\lambda)$
be the unit sphere in $\bigoplus_{k\in A(\lambda)}\V(\lambda_k)\subset \V^{+}$. Recall that $m_{\lambda}|_{\cS^+}$ is a homeomorphism from $\cS^+$ to $\cN_\lambda$. 
Since  $J_\lambda$ is even,  $m_{\lambda}$ is odd. Similarly as in Lemma \ref{beta} we show that for $u\in S(\lambda)$
$$J_\lambda(m_\lambda(u))\leq  \max_{k\in A(\lambda)}\frac13(\lambda+\lambda_{k})^{\frac{3}{2}}|\Om|=:\beta$$
and thus 
$m_\lambda(S(\lambda))\subset J_\lambda^{-1}((0,\beta])\cap \cN_\lambda$.
Hence
$$\gamma(J_\lambda^{-1}(0,\beta])\cap \cN_\lambda)\geq 
\gamma(S(\lambda))=\wt{m}_\lambda.$$
Since $\lambda <-\lambda_k +\frac13\ol S_\curl(\Om)|\Om|^{-\frac23}$ (cf. Lemma \ref{beta}), we have $\beta<c_0$ and it follows that
$m(\cN_\lambda,c_0)\geq \wt{M}(\lambda)$ which completes the proof.
\end{altproof}

\section{Open problems} \label{op}

In this section we state some open problems. Some of them have already been mentioned earlier. 

\begin{itemize}
	\item[(P1)] Does there exist a ground state solution $u$ whose support is a proper subset of $\r3$? In particular, can a ground state have compact support? 
	\item[(P2)]  Can one find an explicit expression for a ground state? Or at least, what can be said about the decay properties of ground states? If they are the same as for the Aubin-Talenti instantons, then one could hopefully retrieve the formulas in the middle of p. 35 in \cite{Willem} which could be useful when looking for ground states for \eqref{eq} with the right-hand side $|u|^4u+g(x,u)$ where $g$ is a monotone lower order term. 
	\item[(P3)] 
	Do the ground state solutions to \eqref{eq} have any symmetry properties? How regular are they?
	\item[(P4)] 
	If $\Om$ is a bounded domain which is neither convex nor has  $\cC^{1,1}$ boundary, then $\cV\subset H^s(\Om,\r3)$ where $s\in[1/2,1]$ and $s$ may be strictly less than 1, see Subsection 2.2 and \cite{Costabel}. Note that the critical exponent for $H^s$ is $6/(3-2s)<6$ if $s<1$. Do the results of Theorem \ref{thm:mainBN} remain valid (with the same right-hand side)? Here the boundary condition \eqref{eq:bc} should be understood in a generalized sense, i.e.  $u$ should be in $W_0^6(\curl;\Om)$.
	\item[(P5)] Can the inequality $S_\curl\geq  \ol S_\curl(\Om)\geq S$ be sharpened? Do there exist domains as in (P4) for which $\ol S_\curl(\Om)<S$?
\end{itemize}

\medskip

{\bf Acknowledgements.}
The authors would like to thank the referee for useful remarks.
J. Mederski was partially supported by the National Science Centre, Poland (Grant No. 2017/26/E/ST1/00817). He was also partially supported by the Alexander von Humboldt Foundation (Germany) and by the Deutsche Forschungsgemeinschaft (DFG, German Research Foundation) - Project ID 258734477 - SFB 1173 during the stay at Karlsruhe Institute of Technology.

{\bf Compliance with Ethical Standards.} The authors declare that they have no conflict of
interests, they also confirm that the manuscript complies to the Ethical Rules applicable for this
journal.

\end{document}